\documentclass[10pt]{article}
 \font\smallit=cmti10

\usepackage{amssymb,latexsym,amsmath,epsfig,amsthm} 
\usepackage{txfonts}
\usepackage{xcolor}
\usepackage{hyperref}
\usepackage{enumerate}
\usepackage{booktabs}
\usepackage{float}
\usepackage{placeins}
\usepackage{authblk}
\usepackage{xurl}

\makeatletter

\renewcommand{\@seccntformat}[1]{\csname the#1\endcsname. }

\makeatother

\theoremstyle{plain}
\newtheorem{theorem}{Theorem}
\newtheorem{lemma}{Lemma}

\newtheorem{definition}{Definition}
\newtheorem{remark}{Remark}
\newtheorem{prop}{Proposition}
\newtheorem{fact}{Fact}
\newtheorem{condition}{Condition}
\newtheorem*{notation}{Notations}
\newtheorem{example}{Example}
\newtheorem*{maintheorem}{Theorem}
\newtheorem*{results}{Previous nonexistence results}

\theoremstyle{definition}
\newtheorem*{structure}{Structure of the paper}

\title{Nonexistence of certain classes of generalized bent functions: Revisiting the element partition method}
\author{Shi Ying and Yingpu Deng}
\date{May 13, 2026}

\begin{document}

\maketitle

\begin{center}

 {\smallit State Key Laboratory of Mathematical Sciences, Academy of Mathematics and Systems Science, Chinese Academy of Sciences, Beijing 100190, People's Republic of China}\\
 {and}\\
 {\smallit School of Mathematical Sciences, University of Chinese Academy of Sciences, Beijing 100049, People's Republic of China}\\

 \vskip 10pt

 {\tt \{yingshi,dengyp\}@amss.ac.cn}\\

 \vskip 10pt

 \noindent\textbf{Corresponding author:} Shi Ying, \texttt{yingshi@amss.ac.cn}

 \end{center}
 \vskip 11pt

\centerline{\bf \large Abstract}

\vskip 5pt

We obtain new nonexistence results for two classes of generalized bent functions
from $\mathbb{Z}_{q}^{n}$ to $\mathbb{Z}_{q}$, called type $[n,q]$ generalized
bent functions. The first class concerns the case
$q=2 p_1^{e_1} p_2^{e_2}$, where $p_1$ and $p_2$ are distinct odd primes.
By applying the element partition method introduced by Lv and Li to earlier results of Feng and Feng-Liu, we obtain sharper nonexistence results for several families of
parameters satisfying explicit congruence and order conditions. These results
extend known nonexistence theorems in cases where the prime divisors of the
odd part of $q$ are self-conjugate. The second class concerns the case
$q=2 \cdot 3^a \cdot 7^b$. By extending the idea of the element partition
method and combining it with explicit computations in suitable cyclotomic
fields and their subfields, we prove that generalized bent functions of type
$[1,2\cdot 3^a\cdot 7^b]$ do not exist for all positive integers $a$ and $b$.

\vskip 5pt


\noindent

\pagestyle{myheadings}

 \thispagestyle{empty}
 \baselineskip=12.875pt
 \vskip 17pt

\section{Introduction} \label{section_introduction}

Bent functions were first introduced by Rothaus in \cite{Rothaus_bent_functions} as Boolean functions having maximum distance to linear functions. Due to their applications in various subjects of coding theory, cryptography and combinatorics, bent functions have been widely studied and there are numerous works. We refer the reader to \cite{Survey_Bent_functions} for a survey of research on bent functions.   

The concept of bent functions has been generalized in many contexts, and we refer the reader to \cite{Survey_GBF} and \cite{Survey_GBF_2} for a survey on various generalizations of bent functions. Similar to bent functions, these generalizations are connected to many subjects in combinatorics, see \cite{GBF_and_combinatorics}, \cite{Survey_GBF_2} and \cite{Schmidt_combinatorics_subjects} for instance.

In this paper, we focus on the following generalization of bent functions, which was introduced by Kumar et al. \cite{def_GBF_1985} in 1985:

\begin{definition}
Let $n$ be a positive integer. Let $q \geq 2$ be an integer, $\mathbb{Z}_{q}=\mathbb{Z} / q \mathbb{Z}$ and $\zeta_{q}=e^{2 \pi \mathrm{i}/q}$. A function $f: \mathbb{Z}_{q}^{n} \rightarrow \mathbb{Z}_q$ is called a generalized bent function (GBF for short) of type $[n,q]$ if $\alpha_{f}(\lambda) \cdot \overline{\alpha_{f}(\lambda)}=q^{n}$ for all $\lambda \in \mathbb{Z}_{q}^{n}$, where
$$\alpha_{f}(\lambda)=\sum_{x \in \mathbb{Z}_{q}^{n}} \zeta_{q}^{f(x)-\lambda \cdot x}$$
is the Fourier coefficient (up to the constant $\sqrt{q^{n}}$, which can be viewed as the ``norm'' of the function $\zeta_{q}^{\lambda \cdot x}$) of $\zeta_{q}^{f(x)}$, and the dot in $\lambda \cdot x$ denotes the canonical dot product.
\end{definition}

Note that when $q=2$ this recovers the original concept of bent functions introduced by Rothaus. In the literature, by a generalized bent function, we always refer to this kind of generalization.

Rothaus \cite{Rothaus_bent_functions} proved that bent functions of length $n$ exist if and only if $2 \ | \ n$, i.e. generalized bent function of type $[n,2]$ exists iff $2 \ | \ n$. Thus, it is natural to ask the same question for generalized bent functions (with $q \neq 2$).

Kumar et al. \cite{def_GBF_1985} constructed generalized bent functions of type $[n,q]$ for $n$ and $q$ satisfying $2 \ | \ n$ or $q \nequiv 2 \ (\mathrm{mod} \ 4)$. On the other hand, when $2 \nmid n$ and $q \equiv 2 \ (\mathrm{mod} \ 4)$, no generalized bent function has been constructed, but many nonexistence results have been proved. It will be convenient to introduce the following long-standing (see Remark 2.2 of \cite{def_self_conj_Schmidt} for example) definition before stating some previous nonexistence results.

\begin{definition}
Let $p$ be a prime and $N=p^{e} N^{\prime}$ with $e \geq 0$ and $p \nmid N^{\prime}$. We say that $p$ is self-conjugate $\mathrm{mod} \ N$ if there exists $s \in \mathbb{Z}$ such that $p^s \equiv -1 \ (\mathrm{mod} \ N^{\prime})$. Note that when $N=p^e$ then $p$ is automatically self-conjugate $\mathrm{mod} \ N$.
\end{definition}

Now we are in a position to state the previous nonexistence results.

\begin{results}
Let $2 \nmid n$ and $q=2N$ with $N$ odd, we list some of the previous nonexistence results as follows:
\begin{enumerate}[(1)]
    \item (\cite[Prop 6]{def_GBF_1985}) Type $[n,2N]$ where $2$ is self-conjugate $\mathrm{mod} \ N$.
    \item \label{result_Ikeda} (\cite[The theorem on page 110]{Ikeda}) Type $[1,2N]$ where every prime divisor of $N$ is self-conjugate $\mathrm{mod} \ N$.
    \item \label{reslut_p_Feng} (\cite[Thm 3.1]{FKQ_earliest}) Type $[n,2 p^{e}]$ where $n<M$ for some upper bound $M$ and $p \equiv 7 \ (\mathrm{mod} \ 8)$ is a prime.
    \item \label{result_Feng} (\cite[Thm 4.1]{FKQ_earliest}) Type $[n,2N]$ where $N=p_{1}^{e_1} p_{2}^{e_2}$, $n < M$ for some upper bound $M$ and $p_1 \equiv 3 \ (\mathrm{mod} \ 4)$, $p_2 \equiv 5 \ (\mathrm{mod} \ 8)$ are primes such that $(\frac{p_1}{p_2})=-1$, which is equivalent to $p_1$, $p_2$ being self-conjugate $\mathrm{mod} \ N$.
    \item \label{result_Feng_Liu} (\cite[Thm 1 and Thm 2]{Feng_Liu}) Type $[n,2N]$ where $N=p_{1}^{e_1}p_{2}^{e_2}$, $n<M$ for some upper bound $M$, $p_1 \equiv 3 \ (\mathrm{mod} \ 4)$, $p_2 \equiv 1 \pmod{8}$, $(\frac{p_1}{p_2})=-1$ and $2$ is not a quartic residue $\mathrm{mod} \ p_2$. Note that in this case, $p_1$ and $p_2$ are also self-conjugate $\mathrm{mod} \ N$.
    \item \label{result_not_conjugate} (\cite{Feng_Liu_2} and \cite{Feng_Liu_Ma}) Type $[n,2N]$ where $N=p_1 p_2$, $n<M$ for some upper bound $M$ and at least one of $p_1$ and $p_2$ is not self-conjugate $\mathrm{mod} \ N$.
    \item \label{result_Jiang_Deng} (\cite[Thm 2]{Jiang_Deng}) Type $[n,2N]$ where $N=p_{1}^{e_1}p_2$, $n<M$ for some upper bound $M$ and at least one of $p_1$ and $p_2$ is not self-conjugate $\mathrm{mod} \ N$, generalizing the above result \ref{result_not_conjugate}.
    \item \label{result_3_23} (\cite[Thm 4]{Jiang_Deng}) Type $[3,2 \cdot 23^e]$, which generalizes the above result \ref{reslut_p_Feng}.
    \item \label{result_Li_Deng} (\cite[Thm 1]{Li_Deng}) See section \ref{section_review_EPM}. Note that this result is also a generalization of the above result \ref{reslut_p_Feng}.
    \item \label{Result_Lv_Li} (\cite{Lv_Li}) Some results generalizing the above results \ref{reslut_p_Feng} and \ref{result_Feng}.
    \item (\cite{Schmidt_nonexistence_results}) Having exactly the same type as result \ref{reslut_p_Feng}, but generalizing all above results \ref{reslut_p_Feng}, \ref{result_3_23}, \ref{result_Li_Deng} and \ref{Result_Lv_Li} sharing the same setting. In particular, a complete summary for $n=3,5,7$ is given in theorem 46.
    \item (\cite{Lv_2025}) Type $[n,2p^{e}]$ where $n$ has an upper bound and $p \equiv 1 \ (\mathrm{mod} \ 8)$ is a prime such that the multiplicative order of $2 \ \mathrm{mod} \ p$ is odd.
\end{enumerate}
\end{results}

Motivated by the above results, especially results \ref{result_Ikeda}, \ref{result_Feng}, \ref{result_Feng_Liu}, \ref{result_Li_Deng} and \ref{Result_Lv_Li}, we prove further nonexistence results. Our first main result is the following theorem:

\begin{maintheorem} (Theorem \ref{thm_two_primes_3_and_5})
Let $N=p_{1}^{e_1} p_{2}^{e_2}$, where $p_1 \equiv 3 \ (\mathrm{mod} \ 8)$ and $p_2 \equiv 5 \ (\mathrm{mod} \ 8)$ are two primes satisfying the following conditions:

\begin{enumerate}
    \item $(\frac{p_1}{p_2})=-1$, which is equivalent to both $p_1$ and $p_2$ are self-conjugate $\mathrm{mod} \ N$.
    \item The multiplicative order of $2$ in ${(\mathbb{Z} / N \mathbb{Z})}^{*}$ is $\frac{\varphi(N)}{2}$, or $s=1$ by the notation of theorem 4.1 of \cite{FKQ_earliest}. Note that this is satisfied iff $2$ is a primitive root of both $p_1$ and $p_2$ and $\gcd((p_1-1) p_{1}^{e_1-1},(p_2-1)p_{2}^{e_2-1})=2$, provided that $p_1$ and $p_2$ are not Wieferich primes. On the other hand, the condition $(\frac{p_1}{p_2})=-1$ shows that $p_1 \nequiv 1 \pmod{p_2}$ and $p_2 \nequiv 1 \pmod{p_1}$, so the condition can be simplified to $\gcd(\frac{p_1-1}{2},p_2-1)=1$. Also note that under this condition, the decomposition field of $2$ in $K=\mathbb{Q}(\zeta_N)$ is $\mathbb{Q}(\sqrt{-p_1 p_2})$.
\end{enumerate}

Let $m$ be the smallest positive integer such that $p_1 A^2+p_2 B^2=2^{m+2}$ has an integral solution $(A,B)$. Then by theorem 4.1 of \cite{FKQ_earliest}, $m$ is odd, and we have that:
 
There is no GBF of type $[m,2N]$.
\end{maintheorem}

The above theorem generalizes Feng's result in \cite{FKQ_earliest}, which is restated as result \ref{result_Feng} above. Theorems \ref{thm_two_primes_1_and_7} and \ref{thm_twoprimes_3and1} share a similar manner to the above theorem, and generalize the results of Feng-Liu in \cite{Feng_Liu} (restated as result \ref{result_Feng_Liu} above). See section \ref{direct_consequence_of_EPM} for the details.

Our second main result is the following theorem:

\begin{maintheorem} (Theorem \ref{main_thm_2})
Let $N=3^{a} \cdot 7^{b}$, where $a,b \in \mathbb{Z}_{>0}$, then generalized bent function of type $[1,2N]$ does not exist.
\end{maintheorem}

See section \ref{section_extending_EPM} for the details of the above theorem.

\begin{structure}
This paper is organized as follows. In section \ref{section_results_in_ANT} we list some results in algebraic number theory that will be used. In section \ref{section_review_EPM} we briefly review the so-called ``element partition method'' and in section \ref{direct_consequence_of_EPM} we prove theorems \ref{thm_two_primes_3_and_5} to \ref{thm_twoprimes_3and1} by applying the element partition method, extending the results of Feng \cite{FKQ_earliest} and Feng-Liu \cite{Feng_Liu}. In section \ref{section_extending_EPM} we extend the element partition method to a case of type $[1,2N]$ where not all prime divisors of $N$ are self-conjugate $\mathrm{mod} \ N$. To be more specific, we prove the nonexistence of type $[1,2N]$ GBFs, where $N=3^{a} \cdot 7^{b}$, $a,b \in \mathbb{Z}_{>0}$. Note that here $7$ is not self-conjugate $\mathrm{mod} \ N$, and there is an algebraic integer having absolute value $\sqrt{2N}$ in $\mathbb{Z}[\zeta_{N}]$ (see section \ref{section_extending_EPM}). For a type $[n,q]$ GBF with $q=2N$, when not all prime divisors of $N$ are self-conjugate $\mathrm{mod} \ N$, previous nonexistence results (restated as results \ref{result_not_conjugate} and \ref{result_Jiang_Deng} above, see \cite{Feng_Liu_2}, \cite{Feng_Liu_Ma} and \cite{Jiang_Deng} for the details) mainly deal with parameters such that there is no algebraic integer with the prescribed absolute value in $\mathbb{Z}[\zeta_{N}]$, and $v_{p}(N)$ cannot tend to infinity for all prime divisors $p$ of $N$ (here $v_{p}(N)$ denotes the maximal positive integer $t$ such that $p^{t} \ | \ N$). To the best of our knowledge, although the prime divisors of $N$ considered ($3$ and $7$) are rather small, our result in section \ref{section_extending_EPM} is the first nonexistence result with parameters failing to meet the above two conditions, provided that not all prime divisors of $N$ are self-conjugate $\mathrm{mod} \ N$. In section \ref{section_scattered_results} we present some scattered results generalizing the results of Feng \cite{FKQ_earliest} and Feng-Liu \cite{Feng_Liu} under the assumption of the GRH, obtained by checking the conditions of the element partition method (see section \ref{section_review_EPM}) in PARI/GP \cite{PARI}. For each of our results, explanations of why they are new are presented after their statements or proofs. Finally, in section \ref{conclusion_and_future_work}, a short conclusion is given. All PARI/GP codes are available at \url{https://github.com/Bluedaydreaming-Y/Codes-for-nonexistence-results-of-generalized-bent-functions}.
\end{structure}

\section{Some results in algebraic number theory} \label{section_results_in_ANT}

Our method of proving nonexistence results of GBFs involves some classical aspects of algebraic number theory, mainly the basic arithmetic (prime decomposition, class group, units etc.) of number fields, especially cyclotomic fields and their subfields. The reader may consult \cite{JanuszAlgebraicNumberFields},\cite{NumberFieldsMarcus} or \cite{IntroductionToCyclotomicFields}, for example. In this section, we list some facts and lemmas with references of their proofs, and prove some results which will be used in the following parts of the paper.

\begin{notation}
Throughout the paper, for a number field $K$, we let $W_K$ denote the roots of unity in $K$, let $U_K$ denote the unit group of $K$ and let $K^{+}=K \cap \mathbb{R}$ when $K$ is a cyclotomic field.
\end{notation}

\begin{fact} \cite[Lemma 1.6]{IntroductionToCyclotomicFields} \label{conjugate_length1}
    If $\alpha$ is an algebraic integer all of whose conjugates have absolute value $1$, then $\alpha$ is a root of unity.
\end{fact}

\begin{remark} \label{cyclo_absolute_value_1}
Note that if $\alpha \in \mathcal{O}_K$ where $K=\mathbb{Q}(\zeta_{N})$ is a cyclotomic field and $\alpha$ has absolute value $1$, then all conjugates of $\alpha$ have absolute value $1$, since $\mathrm{Gal}(K / \mathbb{Q})$ is abelian and each $\sigma \in \mathrm{Gal}(K / \mathbb{Q})$ commutes with the map of conjugation.
\end{remark}

\begin{fact} \cite[Corollary 3 of thm 3]{NumberFieldsMarcus} \label{roots_of_unity_in_cyclotomic_fields}
    Let $N$ be an odd integer. Then the roots of unity in $\mathbb{Q}(\zeta_{N})$ are the $2N$-th roots of unity.
\end{fact}

\begin{fact} \cite[Corollary 4.13]{IntroductionToCyclotomicFields} \label{[A_2:A_1]}
    Let $K=\mathbb{Q}(\zeta_{N})$ with $N \nequiv 2 \ (\mathrm{mod} \ 4)$ and $K^{+}=K \cap \mathbb{R}$. Then $[U_{K}:W_K  U_{K^{+}}]=1$ if $N$ is a prime power and $[U_{K}:W_K U_{K^{+}}]=2$ if $N$ is not a prime power.
\end{fact}

\begin{fact} \cite[Thm 26]{NumberFieldsMarcus} \label{factorization_rules_of_primes_in_cyclotomic_fields}
    Let $N \nequiv 2 \ (\mathrm{mod} \ 4)$ be an integer and $K=\mathbb{Q}(\zeta_N)$. Let $\varphi$ denote the Euler function, with $\varphi(1)=1$. Let $p$ be a prime and $N=p^{l} N^{\prime}$ with $p \nmid N^{\prime}$, $l \geq 0$. Let $f$ be the multiplicative order of $p$ in ${(\mathbb{Z} / N^{\prime} \mathbb{Z})}^{*}$, i.e. the smallest positive integer $t$ such that $p^{t} \equiv 1 \ (\mathrm{mod} \ N^{\prime})$. Then $p \mathcal{O}_K=(\mathfrak{p}_1 \mathfrak{p}_2 \cdots \mathfrak{p}_g)^{e}$, where $\mathfrak{p}_1,\mathfrak{p}_2,\cdots,\mathfrak{p}_g$ are distinct prime ideals in $\mathcal{O}_K$, each having inertial degree $f$; $e=\varphi(p^l)$ and $g=\frac{\varphi(N^{\prime})}{f}$.
\end{fact}

\begin{fact}\cite[p 45, thm 9.3]{JanuszAlgebraicNumberFields} \label{QuadraticContained}
    Let $p$ be an odd prime. Then $\mathbb{Q}(\sqrt{p}) \subseteq \mathbb{Q}(\zeta_p)$ if $p \equiv 1 \ (\mathrm{mod} \ 4)$; and $\mathbb{Q}(\sqrt{-p}) \subseteq \mathbb{Q}(\zeta_p)$ if $p \equiv 3 \ (\mathrm{mod} \ 4)$.
\end{fact}

\begin{fact} \cite[p 75, thm 1]{number_theory_Ireland_Rosen} \label{fact_quadratic_represented_by_root_of_unity}
Let $p$ be an odd prime and $\zeta_{p}=e^{2 \pi \mathrm{i}/p}$. Then $\displaystyle\sum_{a \in {(\mathbb{Z} / p \mathbb{Z})}^{*}} \Big(\frac{a}{p}\Big) \cdot \zeta_{p}^{a}=\sqrt{p}$ if $p \equiv 1 \ (\mathrm{mod} \ 4)$ and $\displaystyle\sum_{a \in {(\mathbb{Z} / p \mathbb{Z})}^{*}} \Big(\frac{a}{p}\Big) \cdot \zeta_{p}^{a}=\mathrm{i} \sqrt{p}$ if $p \equiv 3 \ (\mathrm{mod} \ 4)$.
\end{fact}

\begin{lemma} \label{lemma_order_mod_p^s}
Assume that $a$ has multiplicative order $f \ \mathrm{mod} \ p$, where $p$ is a prime. If $a^{p-1} \nequiv 1 \ (\mathrm{mod} \ p^2)$, then $a$ has multiplicative order $f \cdot p^{e-1} \ \mathrm{mod} \ p^e$.
\end{lemma}

\begin{proof}
We first consider the order of $a \ \mathrm{mod} \ p^2$, which we denote by $n_{2}$. Since $a$ has order $f \ \mathrm{mod} \ p$, we have that $a^{f}=1+lp$, so $a^{pf}={(1+lp)}^{p} \equiv 1 \ (\mathrm{mod} \ p^2)$, and hence $n_2 \ | \ pf$. On the other hand we have that $a^{n_{2}} \equiv 1 \ (\mathrm{mod} \ p^2)$, and hence $a^{n_{2}} \equiv 1 \ (\mathrm{mod} \ p)$, so $f \ | \ n_{2}$. Hence $n_{2}=f \ \text{or} \ pf$. But since $a^{p-1} \nequiv 1 \ (\mathrm{mod} \ p^2)$, we have that $a^{f} \nequiv 1 \ (\mathrm{mod} \ p^2)$, so $n_2=pf$. This shows that the lemma is true for $e=2$.

Now we prove the lemma for general $e \geq 3$ by induction. Let $n_{e}$,$n_{e-1}$ and $n_{e-2}$ denote the order of $a \ \mathrm{mod} \ p^{e}$, $p^{e-1}$ and $p^{e-2}$ respectively. Suppose the lemma is true for $k \leq e-1$, then $n_{e-1}=f \cdot p^{e-2}=p \cdot n_{e-2}$ and $a^{n_{e-1}}=1+t p^{e-1}$. Then $a^{p \cdot n_{e-1}} \equiv 1 \ (\mathrm{mod} \ p^{e})$, so $n_{e} \ | \ p \cdot n_{e-1}$. Moreover, $n_{e-1} \ | \ n_e$ since $a^{n_e} \equiv 1 \ (\mathrm{mod} \ p^{e-1})$. Hence $n_{e}=n_{e-1} \ \text{or} \ p \cdot n_{e-1}$. Let $a^{n_{e-2}}=1+ s \cdot p^{e-2}$, then since $n_{e-1} > n_{e-2}$ we have that $p \nmid s$, so $a^{p \cdot n_{e-2}}=a^{n_{e-1}}={(1+ s \cdot p^{e-2})}^{p} \equiv 1+s \cdot p^{e-1} \ (\mathrm{mod} \ p^{e})$, and hence $a^{n_{e-1}} \nequiv 1 \ (\mathrm{mod} \ p^e)$, and this shows that $n_e \neq n_{e-1}$. Hence $n_e=p \cdot n_{e-1}$, which completes the proof of the lemma.
\end{proof}

\begin{remark} \label{Wieferich_primes_are_rare}
Note that numbers $a$ satisfying $a^{p-1} \equiv 1 \ (\mathrm{mod} \ p^2)$ are quite rare, see \cite{Wieferich_primes} or \cite{Wieferich_primes_2} for instance. For example, when $a=2$, such primes are called Wieferich primes, and it has been established in \cite{Wieferich_primes_2} that $p=1093$ and $p=3511$ are the only two Wieferich primes less than $6.7 \times 10^{15}$.
\end{remark}

\begin{lemma} \cite[Lemma 4 and 5]{Ikeda} \label{DecompositionGroup} 
    Let $N=p^e N^{\prime}$ be an integer with $p$ prime and $p \nmid N^{\prime}$, $e \geq 0$. Then the decomposition group of $p$ in $\mathrm{Gal}(\mathbb{Q}(\zeta_N) / \mathbb{Q})$ contains the complex conjugation map if and only if there exists an integer $s$ such that $p^{s} \equiv -1 \ (\mathrm{mod} \ N^{\prime})$, or in other words, if and only if $p$ is self-conjugate $\mathrm{mod} \ N$.
\end{lemma}

\begin{lemma} \label{lemma_self_conj_reduce_one_prime_factor}
Let $N$ be an odd integer, $N=p^{e}N^{\prime}$, where $e>0$, $p$ is a prime and $p \nmid N^{\prime}$. Let $K=\mathbb{Q}(\zeta_N)$. Then $\gamma:=\sqrt{(-1)^{\frac{p-1}{2}} p} \in \mathcal{O}_K$ by fact \ref{QuadraticContained}. Assume that $p$ is self-conjugate $\mathrm{mod} \ N$. Let $\alpha \in \mathcal{O}_K$ such that $\alpha \overline{\alpha}={(2N)}^{n}$ where $n \geq 1$. Then $\beta:=\alpha / {\gamma}^{ne} \in \mathcal{O}_K$ and $\beta \overline{\beta}={(2N^{\prime})}^{n}$.
\end{lemma}

\begin{proof}
See page 118 of \cite{Ikeda}. Note that while \cite{Ikeda} only proves the case of $n=1$, the general case is essentially the same.
\end{proof}

When all prime factors of $N$ are self-conjugate $\mathrm{mod} \ N$, we have the following lemma:

\begin{lemma} \label{lemma_self_conj_reduce_to_2}
Let $N=p_{1}^{e_1} p_{2}^{e_2} \cdots p_{s}^{e_s}$, where each $p_i$ is self-conjugate $\mathrm{mod} \ N$. Let $K=\mathbb{Q}(\zeta_N)$ and $\omega_N=\prod_{i=1}^{s} {\Big(\sqrt{(-1)^{\epsilon_i}p_i} \Big)}^{e_i}$, where $\epsilon_i=\frac{p_i-1}{2}$. Then $\omega_N \in \mathcal{O}_K$, $\omega_N \overline{\omega_N}=N$, $(\omega_N)=\overline{(\omega_N)}$ and ${(\omega_N)}^{2}= (\omega_N) \overline{(\omega_N)} =N \mathcal{O}_K$. Let $\alpha \in \mathcal{O}_K$ such that $\alpha \overline{\alpha}={(2N)}^{n}$, where $n \geq 1$. Then $\beta:=\frac{\alpha}{\omega_{N}^{n}} \in \mathcal{O}_K$ and $\beta \overline{\beta}=2^n$.
\end{lemma}

\begin{proof}
See \cite{Ikeda}, page 111 (b) and page 118.
\end{proof}

\begin{lemma} \label{lemma_(1+rootofunity)_divides_prime_above_2}
Let $N$ be an odd integer and $K=\mathbb{Q}(\zeta_{N})$. Let $w$ be a root of unity in $K$. Then $w$ must be a $2N$-th root of unity by fact \ref{roots_of_unity_in_cyclotomic_fields}. If $1+w$ divides some prime ideal lying above $2$, then $w=\pm 1$.
\end{lemma}

\begin{proof}
See page 115 of \cite{Ikeda}.
\end{proof}

\begin{lemma} \label{criteria_for_A2}
    Let $K=\mathbb{Q}(\zeta_N)$ and $K^{+}=K \cap \mathbb{R}$, where $N$ is odd and has at least two distinct prime divisors. Let $u$ be a unit in $K^{+}$ and let $x=(1-\zeta_{N})(1-\zeta_{N}^{-1}) \in K^{+}$. Then $u=v \overline{v}$ for some unit $v$ in $K$ iff $\sqrt{u} \in K^{+}$ or $\sqrt{ux} \in K^{+}$.
\end{lemma}

\begin{proof}
Let $U_{K}$ denote the unit group of $\mathcal{O}_{K}$ and let $U_{K^{+}}$ denote the unit group of $\mathcal{O}_{K^{+}}$. Then $[U_{K}:W_K U_{K^{+}}]=2$ by fact \ref{[A_2:A_1]}, and for simplicity we write $W_K=W$. Let $A_{1}=\{ u^2 \ | \ u \in U_{K^{+}} \}$ and $A_{2}=\{ v \overline{v} \ | \ v \in U_{K} \}$. Consider $\mathrm{N}: U_K \rightarrow A_{2}$ which maps $v \in U_K$ to $v \overline{v}$, i.e. the restriction of the norm map $\mathrm{N}_{K / K^{+}}$ on $U_K$. Then $\mathrm{N}$ is surjective and $WU_{K^{+}} \subseteq \mathrm{N}^{-1}(A_1)$. On the other hand, suppose $y \in \mathrm{N}^{-1}(A_1)$, then $y \overline{y}=u^2$ for some $u \in U_{K^{+}}$, so $\frac{y}{u} \cdot \overline{(\frac{y}{u})}=1$ and $\frac{y}{u}$ is a root of unity by fact \ref{conjugate_length1} (and remark \ref{cyclo_absolute_value_1}, which we omit in the following literature). Thus $y \in WU_{K^{+}}$ and $\mathrm{N}^{-1}(A_1)=WU_{K^{+}}$. It is easy to verify that $\ker{N}=W$ using fact \ref{conjugate_length1}, so $\mathrm{N}$ induces isomorphisms $U_{K} / W \simeq A_{2}$ and $WU_{K^{+}} / W \simeq A_1$. Hence $[A_{2} : A_{1}]=[U_K : WU_{K^{+}}]=2$. By \cite[Prop 2.8]{IntroductionToCyclotomicFields}, $1-\zeta_{N}$ is a unit in $K$, so $x=(1-\zeta_{N})(1-\zeta_{N}^{-1})$ is a unit in $K^{+}$. Moreover, $x=(1-\zeta_{N})(1-\zeta_{N}^{-1})=(-\zeta_{N}^{-1}) \cdot {(1-\zeta_{N})}^{2}$, so it is not a square since if it were, then $-\zeta_{N}^{-1}$ would be a square. But $-\zeta_{N}^{-1}$ is a primitive $2N$-th root of unity, so there would exist a primitive $4N$-th root of unity in $K$, a contradiction to fact \ref{roots_of_unity_in_cyclotomic_fields}. Hence $A_{2}=A_{1} \cup x A_{1}$ and it follows that $u=v \overline{v}$ implies that $\sqrt{u} \in K^{+}$ or $\sqrt{ux} \in K^{+}$. On the other hand, if $u \in U_{K^{+}}$ and $\sqrt{u} \in K^{+}$, then it follows that $\sqrt{u} \in U_{K^{+}}$ and $u \in A_{1} \subseteq A_{2}$. A similar argument applies for the case $\sqrt{ux} \in K^{+}$ and this completes the proof of the lemma. 
\end{proof}

\begin{remark} \label{remark_construction_of_v}
Let the notations be the same as lemma \ref{criteria_for_A2}. Then if $u \in U_{K^{+}}$ and $u=v \overline{v}$ for some $v \in U_K$, one may take $v=\sqrt{u}$ if $\sqrt{u} \in K^{+}$ and take $v=\sqrt{u(-\zeta_{N})}$ if $\sqrt{ux} \in K^{+}$. 
\end{remark}

\begin{proof}
It is obvious that one may take $v=\sqrt{u}$ if $\sqrt{u} \in K^{+}$. On the other hand, if $\sqrt{ux}=\sqrt{u \cdot (-\zeta_{N}^{-1}) \cdot {(1-\zeta_{N})}^{2}} \in K^{+}$, then $\sqrt{u \cdot (-\zeta_{N}^{-1})} \in U_{K}$, since it is in $\mathcal{O}_K$ and is a unit. Hence $\sqrt{u \cdot (-\zeta_{N})} \in U_{K}$. Let $v=\sqrt{u(-\zeta_{N})}$ and it suffices to show that $v \overline{v}=u$. Note that $v^2=-u \zeta_{N}$, so $\overline{v}^{2}=-u \zeta_{N}^{-1}$ and ${(v \overline{v})}^2=u^2$, hence $u=v \overline{v}$ or $u=-v \overline{v}$. Since we already know that $u \in \mathbb{R}$ is the norm of some unit in $K$, and hence is positive, we see that $u=v \overline{v}$ is the only possibility, which completes the proof.
\end{proof}

\begin{lemma} \label{lemma_solution_is_unique_up_to_roots_of_unity}
Let $K=\mathbb{Q}(\zeta_{N})$ be an arbitrary cyclotomic field and $K^{+}=K \cap \mathbb{R}$. Let $u$ be a unit in $K^{+}$ such that $u=v \overline{v}$ for some unit $v \in U_K$. Let $x \in U_K$. Then $x \overline{x}=u$ iff $x=v \cdot w$ for some $w \in W_K$.
\end{lemma}

\begin{proof}
It is easy to verify that if $x=v \cdot w$ then $x \overline{x}=u$. On the other hand, if $x \overline{x}=u=v \overline{v}$, then $\frac{x}{v} \cdot \overline{(\frac{x}{v})}=1$. Since $\frac{x}{v} \in \mathcal{O}_K$, by fact \ref{conjugate_length1} we see that $x=v \cdot w$ for some $w \in W_K$, which completes the proof.
\end{proof}

\begin{lemma} \cite[Lemma 1]{Ikeda} \label{lemma_dot_product_of_Fouriercoefficients}
Let $f$ be a type $[n,q]$ GBF. Then $\forall \ 0 \neq v \in \mathbb{Z}_{q}^{n}$, $\displaystyle\sum_{\lambda \in \mathbb{Z}_{q}^{n}} \alpha_{f}(\lambda) \cdot \overline{\alpha_{f}(\lambda+v)}=0$.
\end{lemma}

\section{A sketch of the element partition method} \label{section_review_EPM}

The idea of the element partition method can be dated back to \cite{Ikeda}. As mentioned in section \ref{section_introduction}, Jiang and Deng \cite{Jiang_Deng} extended the idea in \cite{Ikeda} to prove the nonexistence of type $[3,2 \cdot 23^{e}]$ GBF, and their result was further extended by Li and Deng in \cite{Li_Deng} to the nonexistence of type $[n, 2p^{e}]$ GBF, where $p$ is a prime such that $p \equiv 7 \ (\mathrm{mod} \ 8)$ and $2$ has multiplicative order $\frac{1}{2} \varphi(p^e)$ in $(\mathbb{Z} / p^{e} \mathbb{Z})^{*}$; $n$ is the order of $[\mathfrak{p}]$ in $\mathrm{Cl}(\mathbb{Q}(\sqrt{-p}))$, where $\mathfrak{p}$ is a prime lying above $2$ in $\mathbb{Q}(\sqrt{-p})$ and $[\mathfrak{p}]$ denotes its ideal class; $e \geq 1$. Note that when $2^{p-1} \nequiv 1 \ (\mathrm{mod} \ p^2)$ (which almost always happens, see remark \ref{Wieferich_primes_are_rare}), the condition of $\mathrm{ord}_{p^{e}}(2)=\frac{1}{2}\varphi(p^e)$ is equivalent to $2$ having multiplicative order $\frac{p-1}{2} \ \mathrm{mod} \ p$ by lemma \ref{lemma_order_mod_p^s}. See remark \ref{remark_density_near_primitive_root} for an analysis of the natural density of such primes under GRH. Lv and Li introduced the name of the element partition method in \cite{Lv_Li} and proved some other nonexistence results. In this section we give a brief review of the element partition method and restate the two conditions to be checked when applying the element partition method. The results in section \ref{direct_consequence_of_EPM} and section \ref{section_scattered_results} are obtained by checking these two conditions.

Let $f: \mathbb{Z}_{q}^{n} \rightarrow \mathbb{Z}_{q}$ be a GBF, where $2 \nmid n$ and $q \equiv 2 \ (\mathrm{mod} \ 4)$. Let $q=2N$ and $K=\mathbb{Q}(\zeta_N)$. Assume that all the prime factors of $N$ are self-conjugate $\mathrm{mod} \ N$. \footnote{Only in section \ref{section_extending_EPM} do we deal with cases where this condition does not hold.}

We will need the following definition in the context.

\begin{definition}
Let $K$ be a number field. For an ideal $I$ in $\mathcal{O}_K$, we define its support to be the set of prime ideals $\mathfrak{p} \subseteq \mathcal{O}_K$ such that $\mathfrak{p} \ | \ I$. For a rational prime $p$, we define the $p$-support of $I$ to be the set of prime ideals $\mathfrak{p} \subseteq \mathcal{O}_K$ such that $\mathfrak{p}$ lies above $p$ and $\mathfrak{p} \ | \ I$.
\end{definition}

The element partition method considers the relations between $\alpha_{f}(\lambda)$ and $\alpha_{f}(\lambda+v)$, where $\alpha_{f}(\lambda)=\displaystyle\sum_{x \in \mathbb{Z}_{q}^{n}} \zeta_{q}^{f(x)-\lambda \cdot x}$, $\lambda \in \mathbb{Z}_{q}^{n}$ and $v \in \mathbb{Z}_{q}^{n}$ has order $2$.  Note that 
$$\alpha_{f}(\lambda)+\alpha_{f}(\lambda+v)=2 \sum_{x \in \mathbb{Z}_{q}^{n}, x \cdot v=0} \zeta_{q}^{f(x)-\lambda \cdot x},$$
so $(\alpha_{f}(\lambda))$ and $(\alpha_{f}(\lambda+v))$ have the same $2$-support. (This is because: For any prime ideal $\mathfrak{p}$ lying above $2$, $\alpha_{f}(\lambda) \in \mathfrak{p}$ iff $\alpha_{f}(\lambda+v) \in \mathfrak{p}$ since $2 \sum_{x \in \mathbb{Z}_{q}^{n}, x \cdot v=0} \zeta_{q}^{f(x)-\lambda \cdot x} \in \mathfrak{p}$.) 

Let $\omega_N$ be as in lemma \ref{lemma_self_conj_reduce_to_2}. Let $\beta(\lambda)=\frac{\alpha_{f}(\lambda)}{\omega_{N}^{n}}$, $\beta(\lambda+v)=\frac{\alpha_{f}(\lambda+v)}{\omega_{N}^{n}}$. Since $\alpha_{f}(\lambda)$ and $\alpha_{f}(\lambda+v)$ are solutions to $x \in \mathcal{O}_K, \ x \overline{x}={(2N)}^n$ and each prime divisor of $N$ is self-conjugate $\mathrm{mod} \ N$, by lemma \ref{lemma_self_conj_reduce_to_2} we have that $(\omega_{N})^{n} \ | \ \big( \alpha_{f}(\lambda) \big)$, $(\omega_{N})^{n} \ | \ \big( \alpha_{f}(\lambda+v) \big)$; $\beta(\lambda) \overline{\beta(\lambda)}=2^n$, $\beta(\lambda+v) \overline{\beta(\lambda+v)}=2^n$. Moreover, $\big(\beta(\lambda)\big)$ and $\big(\beta(\lambda+v)\big)$ have same $2$-support since $\big(\alpha_{f}(\lambda)\big)$ and $\big(\alpha_{f}(\lambda+v)\big)$ does. Therefore, if the following condition 

\begin{condition} \label{condition_1}
$\gamma_1, \gamma_2 \in \mathcal{O}_K$; $\gamma_1 \overline{\gamma_1}=\gamma_2 \overline{\gamma_2}=2^{n}$; $(\gamma_1)$ and $(\gamma_2)$ have the same support $\implies$ $(\gamma_1)=(\gamma_2)$ 
\end{condition}

holds, then $(\alpha_{f}(\lambda))=(\alpha_{f}(\lambda+v))$ for all $\lambda$ and $v$.

Assuming condition \ref{condition_1}, we have that $\alpha_{f}(\lambda)=u \cdot \alpha_{f}(\lambda+v)$ for some unit $u \in U_K$. Then since $\alpha_{f}(\lambda)$ and $\alpha_{f}(\lambda+v)$ have the same absolute value, by fact \ref{conjugate_length1} we have that $u$ is a root of unity. Since $2 \ | \ \Big( \alpha_{f}(\lambda)+\alpha_{f}(\lambda+v) \Big)=(1+u) \cdot \alpha_{f}(\lambda)$, if the following condition

\begin{condition} \label{condition_2}
$\gamma \in \mathcal{O}_K$, $\gamma \overline{\gamma}=2^{n}$ $\implies$ $2 \nmid \gamma$
\end{condition}

holds, then $2 \nmid \beta(\lambda)$, so $2 \nmid \alpha_{f}(\lambda)$, and hence $1+u$ must divide some prime ideal lying above $2$. By lemma \ref{lemma_(1+rootofunity)_divides_prime_above_2}, this implies that $u=\pm 1$. Thus we have the following important proposition:

\begin{prop}
For all $\lambda \in \mathbb{Z}_{q}^{n}$ and $v \in \mathbb{Z}_{q}^{n}$ of order $2$, $\alpha_{f}(\lambda)=\pm \alpha_{f}(\lambda+v)$.
\end{prop}

When this proposition and condition \ref{condition_2} hold, we may deduce a contradiction concerning the size of the set $\{ \lambda \in \mathbb{Z}_{q}^{n} \ | \ \alpha_{f}(\lambda)=\alpha_{f}(\lambda+v), \forall \ v \ \text{of order} \ 2 \}$, and thus prove the nonexistence of type $[n,q]$ GBF. See \cite{Li_Deng} or \cite{Lv_Li} for the details.

Therefore, in order to apply the element partition method to prove nonexistence results of type $[n,q]$ GBF, it suffices to check that condition \ref{condition_1} and \ref{condition_2} hold for the corresponding $n$ and $q$.

\begin{remark} \label{remark_min_solution}
Note that if there exists $\alpha \in \mathcal{O}_K$ such that $\alpha \overline{\alpha}=2^{n}$, then $(2 \alpha)(2 \overline{\alpha})=2^{n+2}$. Thus, the maximal odd $n$ such that condition \ref{condition_2} holds is the minimal odd $n$ such that $\alpha \overline{\alpha}=2^n$ for some $\alpha \in \mathcal{O}_K$. (Note that for $t$ strictly smaller than the number mentioned above, there does not exist $\alpha \in \mathcal{O}_K$ such that $\alpha \overline{\alpha}=2^t$, and hence the LHS of condition \ref{condition_2} is empty for such $t$.)
\end{remark}

\section{Some nonexistence results based on the element partition method} \label{direct_consequence_of_EPM}

As mentioned in section \ref{section_review_EPM}, we prove some nonexistence results by checking condition \ref{condition_1} and \ref{condition_2} in this section.

While we omit the exact value of the upper bound $M$ when we state result \ref{result_Feng} in section \ref{section_introduction}, we note that, when $p_1 \equiv 3 \ (\mathrm{mod} \ 8)$, following the notations of theorem 4.1 of \cite{FKQ_earliest}, let $m$ be the smallest positive integer such that $p_1 A^2+p_2 B^2=2^{m+2}$ has an integral solution $(A,B)$, $f$ the order of $2 \ \mathrm{mod} \ N$, $g=\frac{\varphi(N)}{f}$ and $s=\frac{g}{2}$, then $M=\frac{m}{s}$. This, together with the result \ref{result_Feng} we stated in section \ref{section_introduction}, gives a complete restatement of case (1), theorem 4.1 of \cite{FKQ_earliest}, which is generalized by the following theorem:

\begin{theorem} \label{thm_two_primes_3_and_5}
Let $N=p_{1}^{e_1} p_{2}^{e_2}$, where $p_1 \equiv 3 \ (\mathrm{mod} \ 8)$ and $p_2 \equiv 5 \ (\mathrm{mod} \ 8)$ are two primes satisfying the following conditions:

\begin{enumerate}
    \item $(\frac{p_1}{p_2})=-1$, which is equivalent to both $p_1$ and $p_2$ are self-conjugate $\mathrm{mod} \ N$.
    \item The multiplicative order of $2$ in ${(\mathbb{Z} / N \mathbb{Z})}^{*}$ is $\frac{\varphi(N)}{2}$, or $s=1$ by the notation of theorem 4.1 of \cite{FKQ_earliest}. Note that this is satisfied iff $2$ is a primitive root of both $p_1$ and $p_2$ and $\gcd((p_1-1) p_{1}^{e_1-1},(p_2-1)p_{2}^{e_2-1})=2$, provided that $p_1$ and $p_2$ are not Wieferich primes. On the other hand, the condition $(\frac{p_1}{p_2})=-1$ shows that $p_1 \nequiv 1 \pmod{p_2}$ and $p_2 \nequiv 1 \pmod{p_1}$, so the condition can be simplified to $\gcd(\frac{p_1-1}{2},p_2-1)=1$. Also note that under this condition, the decomposition field of $2$ in $K=\mathbb{Q}(\zeta_N)$ is $\mathbb{Q}(\sqrt{-p_1 p_2})$.
\end{enumerate}

Let $m$ be the smallest positive integer such that $p_1 A^2+p_2 B^2=2^{m+2}$ has an integral solution $(A,B)$. Then by theorem 4.1 of \cite{FKQ_earliest}, $m$ is odd, and we have that:
 
There is no GBF of type $[m,2N]$.
\end{theorem}

\begin{remark} \label{remark_infinite_set_of_primes}
Assume GRH, then there are infinitely many primes $\equiv 3 \ (\text{resp.} \ 5) \pmod{8}$ and having $2$ as a primitive root, see \cite{primitiveroot}. Actually, let $A=\prod_{p} (1-\frac{1}{p(p-1)})$ be the Artin constant, then \cite{primitiveroot} shows that assuming GRH, then primes $\equiv 3 \ (\text{resp.} \ 5) \pmod{8}$ and having $2$ as a primitive root has natural density $\frac{A}{2}$, and therefore has relative density $2A \approx 0.748$ in the set of primes $\equiv 3 \ (\text{resp.} \ 5) \pmod{8}$. For the above second condition to hold, excluding the condition that $p_1$ and $p_2$ are not Wieferich primes (which is almost always satisfied), we still need $\gcd(\frac{p_1-1}{2},p_2-1)=1$, which appears far from rare. Hence we may expect that the second condition is satisfied in a large number of cases already satisfying the first condition, which is confirmed by our numerical experiment, see table \ref{tab:theorem1-numerical-results} in the appendix.
\end{remark}

\begin{remark} \label{remark_why_generalize_Feng}
The above theorem generalizes case (1), theorem 4.1 of \cite{FKQ_earliest} in the following sense:

Compared to case (1), theorem 4.1 of \cite{FKQ_earliest}, our theorem only assumes one more condition, that is $2$ having multiplicative order $\frac{\varphi(N)}{2}$. This means that $s=1$, and in this case theorem 4.1 of \cite{FKQ_earliest} states that there is no GBF of type $[n,2N]$ for all $n<m, 2 \nmid n$, while our theorem proves the nonexistence of type $[m,2N]$ GBF. While the case of $s=1$ seems to happen not infrequently (see the above remark), for some scattered results generalizing case (1), theorem 4.1 of \cite{FKQ_earliest} without requiring the order of $2$ to be $\frac{\varphi(N)}{2}$, see section \ref{section_scattered_results}.
\end{remark}

\begin{proof}
It suffices to check that condition \ref{condition_1} and condition \ref{condition_2} hold in this case. We do this, in essential, by giving $m$ another characterization. Let $K=\mathbb{Q}(\zeta_N)$, $E=\mathbb{Q}(\sqrt{-p_1 p_2})$ and $T=\mathbb{Q}(\sqrt{-p_1},\sqrt{p_2}) \subseteq K$. We prove that: $m$ is the smallest positive odd integer such that there exists $\alpha \in \mathcal{O}_K$ satisfying $\alpha \overline{\alpha}=2^{m}$.

Let $\lambda$ be the smallest positive odd integer such that there exists $\alpha \in \mathcal{O}_K$ satisfying $\alpha \overline{\alpha}=2^{\lambda}$. It suffices to show that $m=\lambda$.

Combining lemma 2.2, theorem 4.1 of \cite{FKQ_earliest} and our assumption of $2$ having order $\frac{\varphi(N)}{2}$, we have the following result:

Let $l$ be a positive integer. Let $x \in \mathcal{O}_K$ and $x \overline{x}=2^l$. Then there exists $y \in \mathcal{O}_T$ such that $y=x \cdot \zeta_{2N}^{j}$ for some $j \in \mathbb{Z}$, $y^2 \in \mathcal{O}_{E}$ and $y \overline{y}=2^l$.

Let $\alpha \in \mathcal{O}_K$ be such that $\alpha \overline{\alpha}=2^{\lambda}$, then by the above result there exists $\beta \in \mathcal{O}_T$ such that $\beta \overline{\beta}=2^{\lambda}$, $\beta^{2} \in \mathcal{O}_E$ and $\beta=\alpha \cdot \zeta_{2N}^{j}$. The condition of $\beta \in \mathcal{O}_T$ and $\beta^{2} \in \mathcal{O}_E$ implies that $\beta=\frac{r \sqrt{-p_1}+s \sqrt{p_2} }{2}$ for some $r,s \in \mathbb{Z}$ with $r+s \equiv 0 \pmod 2$, see the proof of theorem 4.1 of \cite{FKQ_earliest}. Then $\beta \overline{\beta}=2^{\lambda}$ implies that $p_1 r^2+p_2 s^2=2^{\lambda+2}$, so $\lambda \geq m$ by the definition of $m$. On the other hand, suppose $p_1 A^2+p_2 B^2=2^{m+2}$, then $A+B \equiv 0 \pmod 2$, so $\alpha=\frac{A \sqrt{-p_1}+ B \sqrt{p_2}}{2} \in \mathcal{O}_K$. We have that $\alpha \overline{\alpha}=2^{m}$, which implies that $m \geq \lambda$ by the definition of $\lambda$. Therefore $m=\lambda$.

Thus for $\alpha \in \mathcal{O}_K$, if $\alpha \overline{\alpha}=2^{m}$, then $2 \nmid \alpha$, since otherwise there would exist some $\beta \in \mathcal{O}_K$ satisfying $\beta \overline{\beta}=2^{m-2}$, contradicting the definition of $m$ (The case of $m=1$ is straightforward). This shows that condition \ref{condition_2} holds in this case.

Since the decomposition field of $2$ is $E=\mathbb{Q}(\sqrt{-p_1 p_2})$, and $2 \mathcal{O}_E=\mathfrak{p} \overline{\mathfrak{p}}$ for some prime ideal $\mathfrak{p}$ in $\mathcal{O}_E$, we have that $2 \mathcal{O}_K=\mathfrak{P} \overline{\mathfrak{P}}$ for some prime ideal $\mathfrak{P}$ in $\mathcal{O}_K$. Therefore for $\alpha \in \mathcal{O}_K$ such that $\alpha \overline{\alpha}=2^{m}$, since $2 \nmid \alpha$, $(\alpha)$ must be ${\mathfrak{P}}^{m}$ or ${\overline{\mathfrak{P}}}^{m}$, which shows that condition \ref{condition_1} is also satisfied, and thus shows that GBF of type $[m,2N]$ does not exist and completes the proof. 
\end{proof}

\begin{example}
These pairs satisfy the assumptions of theorem \ref{thm_two_primes_3_and_5}: $(67,5)$, $(83,5)$, $(11,13)$, $(59,13),(59,37),(19,53)$. The values of $m$ are $9,5,5,11,21,15$ respectively. So by theorem \ref{thm_two_primes_3_and_5} we have the nonexistence results for $[n,2N]$ where $N=p_{1}^{e_1} p_{2}^{e_2}$ and $(n,p_1,p_2)$ are the following pairs:
$$(9,67,5), (5,83,5), (5,11,13), (11,59,13),(21,59,37),(15,19,53).$$
\end{example}

\begin{remark}
The above arguments apply for case (2) of theorem 4.1 of \cite{FKQ_earliest} too, and this gives another proof of theorem 2 of \cite{Lv_Li}. 
\end{remark}

Theorems 1 and 2 of \cite{Feng_Liu} can be generalized in a very similar manner. As per previous, these two theorems are stated as result \ref{result_Feng_Liu} in section \ref{section_introduction} with the exact value of the upper bound $M$ omitted. We shall describe this value in detail before generalizing these two theorems.

Let all the settings stated in result \ref{result_Feng_Liu} remain unchanged, except now we assume that $p_1 \equiv 7 \ (\mathrm{mod} \ 8)$. Let $f$ be the order of $2 \ \mathrm{mod} \ N$, and let $g=\frac{\varphi(N)}{f}$. Then $s=\frac{g}{4}$ is odd, see \cite{Feng_Liu}. Let $m_1$ be the smallest positive integer such that the equation
$$2^{m_1+2}=x_{1}^{2}+p_1 y_{1}^{2}$$ has a $\mathbb{Z}$-solution $(x_1,y_1)$ and let $m_2$ be the smallest positive integer such that the equation $$2^{m_2+2}=x_{2}^{2}+p_1 p_2 \cdot y_{2}^{2}$$ has a $\mathbb{Z}$-solution $(x_2,y_2)$. Let $L=\min \{ m_1,m_2 \}$. Then the upper bound $M=\frac{L}{s}$. This, together with the result \ref{result_Feng_Liu} we stated in section \ref{section_introduction}, gives a complete restatement of theorem 1 of \cite{Feng_Liu}, which is generalized by the following theorem:

\begin{theorem} \label{thm_two_primes_1_and_7}
Let $N=p_{1}^{e_1} p_{2}^{e_2}$, where $p_1 \equiv 7 \ (\mathrm{mod} \ 8)$, $p_2 \equiv 1 \pmod{8}$ are two primes satisfying the following conditions:

\begin{enumerate}
    \item $(\frac{p_1}{p_2})=-1$, which is equivalent to both $p_1$ and $p_2$ are self-conjugate $\mathrm{mod} \ N$.
    \item The multiplicative order of $2$ in ${(\mathbb{Z} / N \mathbb{Z})}^{*}$ is $\frac{\varphi(N)}{4}$, or $s=1$ by the notation of theorem 1 of \cite{Feng_Liu}. Similar to the analysis in theorem \ref{thm_two_primes_3_and_5}, this is equivalent to $\mathrm{ord}_{p_1}(2)=\frac{p_1-1}{2}$, $\mathrm{ord}_{p_2}(2)=\frac{p_2-1}{2}$ and $\gcd(\frac{p_1-1}{2},p_2-1)=1$, provided that $p_1$ and $p_2$ are not Wieferich primes. Also note that under this condition, the decomposition field of $2$ in $K=\mathbb{Q}(\zeta_{N})$ is $\mathbb{Q}(\sqrt{-p_1},\sqrt{p_2})$.
\end{enumerate}

Let $L=\min \{m_1,m_2 \}$ be as above in the description of the upper bound $M$ of theorem 1 of \cite{Feng_Liu}. Then by the same theorem, $L$ is odd, and we have that:
 
There is no GBF of type $[L,2N]$.
\end{theorem}

\begin{remark} \label{remark_density_near_primitive_root}
With the assumption of GRH, proposition 1.12 of \cite{Near_primitive_root} shows that the set of primes $p$ with $p \equiv 7 \pmod{8}$ and $\mathrm{ord}_{p}(2)=\frac{p-1}{2}$ has natural density $\frac{A}{2}$, and therefore has relative density $2A \approx 0.748$ in the set of primes $\equiv 7 \pmod{8}$. Meanwhile, let $S_1$ denote the set of primes $p$ with $p \equiv 1 \pmod{8}$ and $\mathrm{ord}_{p}(2)=\frac{p-1}{2}$, then by theorem 3.1 of \cite{Near_primitive_root}, $S_1$ has natural density $\frac{A}{4}$ under GRH, and therefore has relative density $A \approx 0.374$ in the set of primes $\equiv 1 \pmod{8}$. See the appendix for a detailed calculation. Note also that by Chebotarev's density theorem (see \cite{Chebotarev} for example), the set of primes $p \equiv 1 \pmod{8}$ such that $2$ is a quartic residue $\mathrm{mod} \ p$ has natural density $\frac{1}{8}$, so the set of primes $p \equiv 1 \pmod{8}$ such that $2$ is not a quartic residue $\mathrm{mod} \ p$ (denoted by $S_{2}$, which is the set of primes satisfying the assumption of theorem 1 of \cite{Feng_Liu}) also has natural density $\frac{1}{8}$, and $S_1$ has relative density $2A$ in $S_2$. Similar to the analysis in remark \ref{remark_infinite_set_of_primes}, we can expect that the above condition 2 is satisfied not infrequently. See table \ref{tab:theorem2-nonquartic-numerical-results} for numerical results.
\end{remark}

\begin{remark}
The above theorem generalizes theorem 1 of \cite{Feng_Liu} in a sense similar to remark \ref{remark_why_generalize_Feng}. More specifically, compared to theorem 1 of \cite{Feng_Liu}, we only added the condition of $s=1$, which already implies the condition that $2$ is not quartic residue $\mathrm{mod} \ p_2$. With the same condition added, theorem 1 of \cite{Feng_Liu} proves the nonexistence result of type $[n,2N]$ with $n<L$ and $2 \nmid n$, while our theorem proves the nonexistence of type $[L,2N]$ GBF. Some scattered results generalizing theorem 1 of \cite{Feng_Liu} without assuming $\mathrm{ord}_{N}(2)=\frac{\varphi(N)}{4}$ are given in section \ref{section_scattered_results}.
\end{remark}

\begin{proof}
The proof shares a similar manner to theorem \ref{thm_two_primes_3_and_5}. We check the conditions \ref{condition_1} and \ref{condition_2} as per previous. Let $K=\mathbb{Q}(\zeta_N)$ and $T=\mathbb{Q}(\sqrt{-p_1},\sqrt{p_2}) \subseteq K$. Note that $T$ is the decomposition field of $2$ in $K$. Suppose $k \in \mathbb{Z}$ is odd, $\alpha \in \mathcal{O}_K$ and $\alpha \overline{\alpha}=2^{k}$. Then similar to the proof of theorem 4.1 of \cite{FKQ_earliest}, there exists $b \in \mathbb{Z}$ such that $\beta=\alpha \cdot \zeta_{2N}^{b}$ satisfies $\beta^{2} \in \mathcal{O}_T$. Moreover, in the proof of theorem 1 of \cite{Feng_Liu}, it is shown that actually $\beta \in \mathcal{O}_T$. It is also shown that $L$ is the smallest odd integer such that $2^{L}=x \overline{x}$ for some $x \in \mathcal{O}_T$, so for each $\alpha \in \mathcal{O}_K$ such that $\alpha \overline{\alpha}=2^{L}$, condition \ref{condition_2} is satisfied.

Note that $\mathrm{Gal}(T / \mathbb{Q})=\{ 1, \sigma, \tau, \sigma \tau \}$, where $\sigma(\sqrt{-p_1})=-\sqrt{-p_1}$, $\sigma(\sqrt{p_2})=\sqrt{p_2}$; $\tau(\sqrt{-p_1})=\sqrt{-p_1}$, $\tau(\sqrt{p_2})=-\sqrt{p_2}$. Then $2 \mathcal{O}_{T}=\mathfrak{p} {\mathfrak{p}}^{\sigma} {\mathfrak{p}}^{\tau} {\mathfrak{p}}^{\sigma \tau}$, where $\mathfrak{p}$ is a prime ideal of $\mathcal{O}_{T}$ lying above $2$. Since $T$ is the decomposition field of $2$ in $K$, all prime ideals lying above $2$ in $\mathcal{O}_{T}$ are inert in $\mathcal{O}_K$, so we have that $2 \mathcal{O}_{K}=\mathfrak{P} {\mathfrak{P}}^{\sigma} {\mathfrak{P}}^{\tau} {\mathfrak{P}}^{\sigma \tau}$, where $\mathfrak{P}=\mathfrak{p} \mathcal{O}_{K}$, ${\mathfrak{P}}^{\sigma}={\mathfrak{p}}^{\sigma} \mathcal{O}_{K}$ and the other two ideals follow.

In the proof of theorem 1 of \cite{Feng_Liu}, it is actually shown that: Suppose $x \in \mathcal{O}_{T}$ and $x \overline{x}=2^{L}$, then let $I=(x)$, we have that $I$ must be one of the following four:
$${\mathfrak{p}}^{L} {({\mathfrak{p}}^{\tau})}^{L}, {\mathfrak{p}}^{L} {({\mathfrak{p}}^{\sigma \tau})}^{L}, {({\mathfrak{p}}^{\sigma})}^{L} {({\mathfrak{p}}^{\tau})}^{L} \text{or} \ {({\mathfrak{p}}^{\sigma})}^{L} {({\mathfrak{p}}^{\sigma \tau})}^{L}.$$
Suppose $\alpha \in \mathcal{O}_K$ and $\alpha \overline{\alpha}=2^{L}$. Consider the corresponding $\beta=\alpha \cdot \zeta_{2N}^{b} \in \mathcal{O}_{T}$ as above, then $\beta \mathcal{O}_{T}$ must be one of the above four possibilities, so $\beta \mathcal{O}_K=\alpha \mathcal{O}_{K}$ must be one of the following four:
$${\mathfrak{P}}^{L} {({\mathfrak{P}}^{\tau})}^{L}, {\mathfrak{P}}^{L} {({\mathfrak{P}}^{\sigma \tau})}^{L}, {({\mathfrak{P}}^{\sigma})}^{L} {({\mathfrak{P}}^{\tau})}^{L} \text{or} \ {({\mathfrak{P}}^{\sigma})}^{L} {({\mathfrak{P}}^{\sigma \tau})}^{L}.$$
Hence for $\alpha_1, \alpha_2 \in \mathcal{O}_K$ with $\alpha_{1} \overline{\alpha_1}=\alpha_2 \overline{\alpha_2}=2^{L}$, if $(\alpha_1) \neq (\alpha_2)$, then $(\alpha_1)$ and $(\alpha_2)$ cannot have the same support, so condition \ref{condition_1} also holds, and this completes the proof of the theorem.
\end{proof}

\begin{example}
These pairs satisfy the assumptions of theorem \ref{thm_two_primes_1_and_7}: $(23,17)$, $(23,97)$, $(47,41)$, $(71,17)$, $(79,17),(79,41)$. The values of $L$ are $3,3,5,7,5,5$ respectively. So by theorem \ref{thm_two_primes_1_and_7} we have the nonexistence results for $[n,2N]$ where $N=p_{1}^{e_1} p_{2}^{e_2}$ and $(n,p_1,p_2)$ are the following pairs:
$$(3,23,17), (3,23,97), (5,47,41), (7,71,17), (5,79,17),(5,79,41).$$
\end{example}

Now we generalize theorem 2 of \cite{Feng_Liu}. As above, we first give a complete restatement of the upper bound in this theorem.

Let all the settings stated in result \ref{result_Feng_Liu} remain unchanged, except now we assume that $p_1 \equiv 3 \ (\mathrm{mod} \ 8)$. Let $f$ be the order of $2 \ \mathrm{mod} \ N$, and let $g=\frac{\varphi(N)}{f}$. Then $s=\frac{g}{4}$ is odd, see \cite{Feng_Liu}. Let $A$ and $B$ be the integers determined by $p_2=A^{2}+B^2$, where $2 \nmid A$. Let $m$ be the smallest positive integer such that the following equations
\begin{equation} \label{Diophantine_equa}
\begin{cases}
2^{m+4}=p_{1}^{2} X^2+p_2 Y^2+2 p_1 p_2 (Z^2+W^2)
\\[0.5 em]
XY+2BZW+A(W^2-Z^2)=0
\end{cases}
\end{equation}
have a $\mathbb{Z}$-solution. Note that by changing $W$ and $Z$, $A$ can be chosen to be $\equiv 1 \ \text{or} \ 3 \pmod{4}$ arbitrarily. Note also that here it is guaranteed that $XY \neq 0$, and reducing $\mathrm{mod} \ p_2$ shows that $m$ must be odd, so some statements in \cite{Feng_Liu} are not needed here. Then the upper bound $M=\frac{m}{s}$. This gives a complete restatement of theorem 2 of \cite{Feng_Liu}, and it is generalized by the following theorem by adding the condition of $s=1$:

\begin{theorem} \label{thm_twoprimes_3and1}
Let $N=p_{1}^{e_1} p_{2}^{e_2}$, where $p_1 \equiv 3 \ (\mathrm{mod} \ 8)$, $p_2 \equiv 1 \pmod{8}$ are two primes satisfying the following conditions:

\begin{enumerate}
    \item $(\frac{p_1}{p_2})=-1$, which is equivalent to both $p_1$ and $p_2$ are self-conjugate $\mathrm{mod} \ N$.
    \item The multiplicative order of $2$ in ${(\mathbb{Z} / N \mathbb{Z})}^{*}$ is $\frac{\varphi(N)}{4}$, or $s=1$ by the notation of theorem 1 of \cite{Feng_Liu}. Similar to the above theorems, this is equivalent to $\mathrm{ord}_{p_1}(2)=p_1-1$, $\mathrm{ord}_{p_2}(2)=\frac{p_2-1}{2}$ and $\gcd(\frac{p_1-1}{2},p_2-1)=1$, provided that $p_1$ and $p_2$ are not Wieferich primes. Also note that under this condition, the decomposition field of $2$ in $K=\mathbb{Q}(\zeta_{N})$ is a cyclic quartic field containing $\mathbb{Q}(\sqrt{p_2})$, see \cite{Feng_Liu}.
\end{enumerate}
Let $m$ have the same meaning as above, i.e. it is the smallest positive integer such that the above Diophantine Equations (\ref{Diophantine_equa}) have a $\mathbb{Z}$-solution. Then $m$ is odd and there is no GBF of type $[m,2N]$.
\end{theorem}

\begin{remark}
Similar as above, assuming GRH, the set of primes $p$ such that $p \equiv 3 \pmod{8}$ and $2$ is a primitive root $\mathrm{mod} \ p$ has natural density $\frac{A}{2}$ with $A$ denoting the Artin's constant, and the set of primes $p$ such that $p \equiv 1 \pmod{8}$ and $\mathrm{ord}_{p}(2)=\frac{p-1}{2}$ has natural density $\frac{A}{4}$, so we may expect that the above condition 2 is satisfied not infrequently. See table \ref{tab:theorem3-nonquartic-numerical-results} for numerical results.
\end{remark}

\begin{remark}
Similar as above, compared to theorem 2 of \cite{Feng_Liu}, we only assumed one more condition regarding the order of $2$ $\mathrm{mod} \ N$. Under this assumption theorem 2 of \cite{Feng_Liu} shows the nonexistence results of type $[n,2N]$ for $2 \nmid n$, $n<m$, while our result shows the nonexistence result of $[m,2N]$. For some scattered results generalizing theorem 2 of \cite{Feng_Liu} without constraints on the order of $2$, see section \ref{section_scattered_results}.
\end{remark}

\begin{proof}
Similar as above, we check the conditions \ref{condition_1} and \ref{condition_2} in section \ref{section_review_EPM}.

Let $K=\mathbb{Q}(\zeta_{N})$. Let $M$ be the unique quartic subfield of $\mathbb{Q}(\zeta_{p_2})$, then $M=\mathbb{Q}(\gamma_0)$, where $\gamma_{0}=\frac{\sqrt{p_2}-1+\sqrt{2 p_2+2 A_{0} \sqrt{p_2}}}{4}$, here $A_{0} \equiv 3 \pmod{4}$ is determined by $p_2=A_{0}^{2}+B^2$, see section 4.2 of \cite{Gauss_and_Jacobi_sums} for the detailed computation on the Gauss sum $g(4)$ included in the process of determining the primitive element $\gamma_0$ of $M$. Let $\tau$ be a generator of $\mathrm{Gal}(M / \mathbb{Q})$, then $\{ \tau^{i}(\gamma_0) \ | \ 0 \leq i \leq 3 \}$ is an integral basis of $\mathcal{O}_M$, see lemma 2.19 of \cite{Feng_Liu}.

Let $H=\mathbb{Q}(\sqrt{-p_1},\gamma_0)$, then an integral basis of $\mathcal{O}_H$ is $\{ \tau^{i}(\gamma_0), \frac{1+\sqrt{-p_1}}{2} \tau^{i}(\gamma_0) \ | \ 0 \leq i \leq 3 \}$. Let $T$ be the decomposition field of $2$ in $K$, then $T/ \mathbb{Q}$ is a cyclic quartic extension, $T \subseteq H$ and $T$ is fixed by the element in $\mathrm{Gal}(H / \mathbb{Q})$ that maps $\sqrt{-p_1}$ to $-\sqrt{-p_1}$ and maps $\gamma_{0}$ to $\tau^{2}(\gamma_0)$. Thus, let $x \in \mathcal{O}_T$, then we may express $x$ as the $\mathbb{Z}$-linear combination of the integral basis of $\mathcal{O}_H$ and use the condition that $x$ is stable under $\mathrm{Gal}(H/T)$ to obtain more restrictions. It turns out that an integral basis of $\mathcal{O}_T$ is:
$$\Big\{ 1, \frac{\sqrt{p_2}-1}{2}, \frac{\sqrt{p_2}-1+\sqrt{-2 p_1 (p_2 +A_0 \sqrt{p_2})}}{4}, \frac{-\sqrt{p_2}-1+\sqrt{-2 p_1 (p_2 - A_0 \sqrt{p_2})}}{4} \Big\}.$$

Let $\beta_{0}=\frac{\sqrt{p_2}-1+\sqrt{-2 p_1 (p_2+A_0 \sqrt{p_2})}}{4}$. Then $T=\mathbb{Q}(\beta_{0})$, and it is straightforward (although a little bit tedious) to compute that $\mathrm{disc}(\beta_{0})=\mathrm{disc}(T)=p_{1}^{2}p_{2}^{3}$. Thus let $\sigma$ be a generator of $\mathrm{Gal}(T / \mathbb{Q})$, then $\{ \sigma^{i}(\beta_{0}) \ | \ 0 \leq i \leq 3 \}$ is an integral basis of $\mathcal{O}_T$. Let $\eta_{0}=\frac{1-\sqrt{p_2}+\sqrt{-2 p_1 (p_2 +A_0 \sqrt{p_2})}}{4}$. Then it is straightforward to verify that $T=\mathbb{Q}(\eta_{0})$ and $\{ \sigma^{i}(\eta_{0}) \ | \ 0 \leq i \leq 3 \}$ is also an integral basis of $\mathcal{O}_T$.

Note that in page 3070 of \cite{Feng_Liu}, $\eta_{0}$ is set to be $\frac{1-\sqrt{p_2}+\sqrt{-2 p_1 (p_2 -A_0 \sqrt{p_2})}}{4}$, where $A_0 \equiv 3 \pmod{4}$ is as above, but with this choice, $\eta_{0}$ would not be an algebraic integer. To make the setting correct and coincide with all the computations in \cite{Feng_Liu}, we set $A$ to be the integer determined by $p_2=A^{2}+B^{2}$ and $A \equiv 1 \pmod{4}$, set $\eta_{0}=\frac{1-\sqrt{p_2}+\sqrt{-2 p_1 (p_2 -A \sqrt{p_2})}}{4}$, i.e. the setting $A \equiv -1 \pmod{4}$ on page 3070 of \cite{Feng_Liu} should be changed into $A \equiv 1 \pmod{4}$, with all other setting preserved, and now everything would be correct. We follow this notation throughout the following proof. Without loss of generality we may assume that $\sigma$ maps $\eta_{0}$ to $\eta_{1}=\frac{1+\sqrt{p_2}+\sqrt{-2 p_1 (p_2+A \sqrt{p_2})}}{4}$, and by our above analysis $\{ \eta_{i} \ | \ 0 \leq i \leq 3 \}$ is an integral basis of $\mathcal{O}_{T}$, where $\eta_{i}=\sigma^{i}(\eta_{0})$.

Let $k$ be a positive odd integer. It is proved in the proof of theorem 2 of \cite{Feng_Liu} that if $\alpha \in \mathcal{O}_K$ and $\alpha \overline{\alpha}=2^{k}$, then there exists $b \in \mathbb{Z}$ such that $\beta=\alpha \cdot \zeta_{2N}^{b} \in \mathcal{O}_T$. Thus, as in \cite{Feng_Liu}, by representing $\beta$ as the $\mathbb{Z}$-linear combination of $\eta_{0}, \eta_{1}, \eta_{2} \ \text{and} \ \eta_{3}$, we see that the equations
\begin{equation} 
\begin{cases}
2^{k+4}=p_{1}^{2} X^2+p_2 Y^2+2 p_1 p_2 (Z^2+W^2)
\\[0.5 em]
XY+2BZW+A(W^2-Z^2)=0
\end{cases}
\end{equation}
has a $\mathbb{Z}$-solution. Now assume that $m \geq 3$, $\alpha \in \mathcal{O}_K$, $\alpha \overline{\alpha}=2^{m}$ and $2 \ | \ \alpha$. Then let $\alpha^{\prime}=\frac{\alpha}{2} \in \mathcal{O}_K$, by our above analysis we see that the equations
\begin{equation} 
\begin{cases}
2^{m+2}=p_{1}^{2} X^2+p_2 Y^2+2 p_1 p_2 (Z^2+W^2)
\\[0.5 em]
XY+2BZW+A(W^2-Z^2)=0
\end{cases}
\end{equation}
would have a $\mathbb{Z}$-solution, a contradiction to the definition of $m$, so it follows that $2 \nmid \alpha$. The case of $m=1$ is straightforward, so we see that condition \ref{condition_2} holds.

Now we deal with condition \ref{condition_1}. This has essentially been done in the remark after the proof of theorem 2 in \cite{Feng_Liu}, for the sake of completeness we restate some vital steps and add some detailed analysis connected to our context here.

Let $\mathfrak{p}_{0}$ be a prime ideal of $\mathcal{O}_T$ lying above $2$, then $2 \mathcal{O}_{T}=\mathfrak{p}_{0} \mathfrak{p}_{1} \mathfrak{p}_{2} \mathfrak{p}_{3}$, where $\mathfrak{p}_{i}=\mathfrak{p}_{0}^{\sigma^{i}}$. Note also that $\sigma^{2}(x)=\overline{x}$ for all $x \in T$. Moreover, since $\mathfrak{p}_i$ is inert in $K$, we have that $2 \mathcal{O}_{K}=\mathfrak{P}_{0} \mathfrak{P}_{1} \mathfrak{P}_{2} \mathfrak{P}_{3}$ with $\mathfrak{P}_{i}=\mathfrak{p}_{i} \mathcal{O}_K$.

Let $l$ be the order of $[\mathfrak{p}_{0} \mathfrak{p}_{1}]$ in $\mathrm{Cl}(T)$, then it is shown in \cite{Feng_Liu} that $l$ is odd. Here we note that, in \cite{Feng_Liu}, it is deduced from ${(\mathfrak{p}_{0} \mathfrak{p}_{1})}^{l}=\alpha \mathcal{O}_{T}$ that $2^{l}=\alpha \overline{\alpha}$, which is in general incorrect, since the direct consequence of ${(\mathfrak{p}_{0} \mathfrak{p}_{1})}^{l}=\alpha \mathcal{O}_{T}$ is $2^{l} \mathcal{O}_{T}=\alpha \overline{\alpha} \mathcal{O}_{T}$, and thus there exists $u \in U_T$ such that $\alpha \overline{\alpha}=2^{l} \cdot u$. However, here we see that $u \in \mathbb{R}_{>0}$, hence $u \in U_{\mathbb{Q}(\sqrt{p_2})}$. Let $v>1$ be the fundamental unit of $\mathbb{Q}(\sqrt{p_2})$, then $u=v^s$ for some $s \in \mathbb{Z}$. Moreover, we have that
$$\mathrm{N}_{\mathbb{Q}(\sqrt{p_2}) / \mathbb{Q}}(\alpha \overline{\alpha})=(\alpha \sigma(\alpha)) \cdot \overline{(\alpha \sigma(\alpha))}=2^{2l} \cdot \mathrm{N}_{\mathbb{Q}(\sqrt{p_2}) / \mathbb{Q}}(u),$$
so $\mathrm{N}_{\mathbb{Q}(\sqrt{p_2}) / \mathbb{Q}}(u) > 0$. It is well-known that (see prop 2.1 of \cite{lemmermeyer2003higherdescentpellconics} for example) if $p$ is a prime such that $p \equiv 1 \pmod{4}$, then the equation $x^2-p y^2=-1$ has $\mathbb{Z}$-solution $(x,y)$, so there exists a unit of $\mathbb{Q}(\sqrt{p_2})$ with norm $-1$, so the fundamental unit $v$ must also have norm $-1$. Hence $s=2s_{0}$ and $u={(v^{s_0})}^{2}$, and let $\beta=\alpha \cdot v^{-s_0} \in \mathcal{O}_{T}$, then we have that $\beta \overline{\beta}=2^{l}$, which is enough to prove that $l$ is odd.

Let $G=\mathrm{Gal}(T / \mathbb{Q})$. Consider the action of $R=\mathbb{Z}[G] \big/ (1+\sigma^{2}) \simeq \mathbb{Z}[\mathrm{i}]$ on the class $[\mathfrak{p}_{0} \mathfrak{p}_{1}]$ and set $I=\{ \pi \in R \ | \ {[\mathfrak{p}_{0} \mathfrak{p}_{1}]}^{\pi}=1 \}$. Then since $R \simeq \mathbb{Z}[\mathrm{i}]$ is a PID, $I=\gamma R$ for some $\gamma=a+b \sigma$ with $a,b \geq 0$. (Here $\sigma$ is actually $\overline{\sigma}$, the image of $\sigma$ under the quotient map, for simplicity we still use the notation $\sigma$.) Then we have that $a+b \mathrm{i} \ | \ l$, so $a^2+b^2 \ | \  l$ and thus $a+b$ is odd since $l$ is. Moreover, in particular, we have that ${(\mathfrak{p}_{0} \mathfrak{p}_{1})}^{a+b \sigma}=\alpha \mathcal{O}_{T}$ for some $\alpha \in \mathcal{O}_{T}$ , so
$$\mathfrak{p}_{0}^{a+(a+b) \sigma+b \sigma^{2}}=\alpha \mathcal{O}_{T}.$$
Taking conjugates shows that
$$\alpha \overline{\alpha} \mathcal{O}_{T}=\mathfrak{p}_{0}^{(a+b)(1+\sigma+\sigma^{2}+\sigma^{3})} \mathcal{O}_{T}=2^{a+b} \mathcal{O}_{T}.$$
Proceed as above, we see that there exists $\beta \in \mathcal{O}_{T}$ such that $\beta \overline{\beta}=2^{a+b}$, and by expressing $\beta$ as the $\mathbb{Z}$-linear combination of $\eta_{i}$, we obtain the corresponding Diophantine equations, so by the definition of $m$ we have that $$a+b \geq m.$$

Let $x \in \mathcal{O}_{T}$ be such that $x \overline{x}=2^{m}$ and let $I=x \mathcal{O}_{T}$. Let $I=\mathfrak{p}_{0}^{\pi}$ where $\pi=a_0 + a_{1} \sigma + a_{2} \sigma^{2} + a_{3} \sigma^{3}$, $a_i \in \mathbb{Z}_{\geq 0}$. Since $2 \nmid x$, at least one of $a_0, a_1, a_2, a_3$ is $0$, and we may first assume that $a_3=0$, then it is shown in \cite{Feng_Liu} that $a_0+a_2=a_1=m$, and $a_0+a_2 \sigma \in I$. Thus we have that $a_{0}+a_{2} \sigma=(a+b \sigma) \cdot r$ for some $r \in R$. Put these into $\mathbb{Z}[\mathrm{i}]$, we have that
$$a_{0}+a_{2} \mathrm{i}=(a+b \mathrm{i}) \cdot \gamma,$$
where $\gamma \in \mathbb{Z}[\mathrm{i}]$. So we have that
$$a_{0}^{2}+a_{2}^{2}=(a^{2}+b^{2}) \cdot {|\gamma|}^{2}.$$
Note also that $a+b \geq m=a_{0}+a_{2}$. Therefore $(a^{2}+b^{2}) \cdot {|\gamma|}^{2}=a_{0}^{2}+a_{2}^{2} \leq {(a_0+a_2)}^{2} \leq {(a+b)}^{2} \leq 2(a^2+b^2)$, so $|\gamma| \leq \sqrt{2}$. However, if $|\gamma|=\sqrt{2}$, then $2 \ | \ a_{0}^{2}+a_{2}^{2}$, so $2 \ | \ a_{0}+a_{2}=m$, a contradiction to $m$ being odd. Thus, $|\gamma| < \sqrt{2}$, so $\gamma$ must be a unit in $\mathbb{Z}[\mathrm{i}]$.

Thus, if $a>0$ and $b>0$, then since $a_0 \geq 0$ and $a_2 \geq 0$, the only possibility is $\gamma=1$, and hence $a_0=a$, $a_2=b$. Now consider the $\pi$ defined above, we see that $\pi$ must be one of the following four:
$$a+m \sigma+ b \sigma^{2}, a \sigma+m \sigma^{2}+ b \sigma^{3}, a \sigma^{2}+m \sigma^{3}+b \ \text{or} \ a \sigma^{3}+m+b \sigma.$$
In terms of $I$, this means that $I$ must be one of the four:
$$\mathfrak{p}_{0}^{a} \mathfrak{p}_{1}^{m} \mathfrak{p}_{2}^{b}, \ \mathfrak{p}_{1}^{a} \mathfrak{p}_{2}^{m} \mathfrak{p}_{3}^{b}, \ \mathfrak{p}_{2}^{a} \mathfrak{p}_{3}^{m} \mathfrak{p}_{0}^{b} \ \ \text{or} \ \ \mathfrak{p}_{3}^{a} \mathfrak{p}_{0}^{m} \mathfrak{p}_{1}^{b}.$$

Thus for $\alpha \in \mathcal{O}_K$ such that $\alpha \overline{\alpha}=2^{m}$, consider the corresponding $\beta=\alpha \cdot \zeta_{2N}^{b}$ with $\beta \in \mathcal{O}_T$, $\beta \overline{\beta}=2^m$, we see that $\alpha \mathcal{O}_K=\beta \mathcal{O}_{K}$ must be one of the following four:
$$\mathfrak{P}_{0}^{a} \mathfrak{P}_{1}^{m} \mathfrak{P}_{2}^{b}, \ \mathfrak{P}_{1}^{a} \mathfrak{P}_{2}^{m} \mathfrak{P}_{3}^{b}, \ \mathfrak{P}_{2}^{a} \mathfrak{P}_{3}^{m} \mathfrak{P}_{0}^{b} \ \ \text{or} \ \ \mathfrak{P}_{3}^{a} \mathfrak{P}_{0}^{m} \mathfrak{P}_{1}^{b}.$$
Hence we see that if $\alpha_1, \alpha_2 \in \mathcal{O}_K$ and $\alpha_i \overline{\alpha_i}=2^m$, then $\alpha_i \mathcal{O}_K$ must be one of the above four ideals, and if $(\alpha_1) \neq (\alpha_2)$ then $(\alpha_1)$ and $(\alpha_2)$ must have different support. This shows that condition \ref{condition_1} holds in this case.

It remains to consider the case of one of $a$ and $b$ is $0$. Assume that $\alpha \in \mathcal{O}_K$ and $\alpha \overline{\alpha}=2^m$. In this case, there are two choices of $\gamma$ (for example when $b=0$, $\gamma$ can be chosen as $1$ or $\mathrm{i}$) and the result concerning $(\alpha)$ is that $(\alpha)$ must be one of the following four:
$$\mathfrak{P}_{0}^{m} \mathfrak{P}_{1}^{m}, \ \mathfrak{P}_{0}^{m} \mathfrak{P}_{3}^{m}, \ \mathfrak{P}_{1}^{m} \mathfrak{P}_{2}^{m} \ \ \text{or} \ \ \mathfrak{P}_{2}^{m} \mathfrak{P}_{3}^{m}.$$

Thus for $\alpha_1, \alpha_2 \in \mathcal{O}_K$ with $\alpha_{i} \overline{\alpha_i}=2^{m}$, $\alpha_i \mathcal{O}_K$ must be one of the above four ideals, so $(\alpha_1) \neq (\alpha_2)$ implies that $(\alpha_1)$ and $(\alpha_2)$ have different support, which again shows that condition \ref{condition_1} holds, and this completes the proof.
\end{proof}

\begin{example}
Pairs $(3,17), (3,41), (11,17), (19,41), (107,17)$ satisfy the assumptions of theorem \ref{thm_twoprimes_3and1}. The values of $m$ are $3,5,5,9,19$ respectively. So by theorem \ref{thm_twoprimes_3and1} we have the nonexistence results for $[n,2N]$ where $N=p_{1}^{e_1} p_{2}^{e_2}$ and $(n,p_1,p_2)$ are the following pairs:
$$(3,3,17), (5,3,41), (5,11,17), (9,19,41), (19,107,17).$$
\end{example}

\section{\texorpdfstring{Nonexistence of type $[1,2 \cdot 3^a \cdot 7^b]$ GBF}{Nonexistence of type [1,2N] GBF with 3 and 7 dividing N}} \label{section_extending_EPM}

In this section, we prove that generalized bent function of type $[1,2 \cdot 3^a \cdot 7^b]$ does not exist, where $a,b \in \mathbb{Z}_{> 0}$. Let $N=3^{a} \cdot 7^{b}$. Note that here $7$ is not self-conjugate $\mathrm{mod} \ N$. We will first consider the case of $a=b=1$, i.e. prove the nonexistence of type $[1,42]$ GBF, then generalize the result to cases of arbitrary $a,b$. Here we give a brief description of our idea.

Let $q \equiv 2 \ (\mathrm{mod} \ 4)$ and let $f$ be a GBF of type $[1,q]$. When $n=1$, let $e=N$, the element partition method considers the size of $A_{1}:=\{ \lambda \in \mathbb{Z} / q \mathbb{Z} \ | \ \alpha_{f}(\lambda)=\alpha_{f}(\lambda+e) \}$ and $A_{2}:=\{ \lambda \in \mathbb{Z} / q \mathbb{Z} \ | \ \alpha_{f}(\lambda)=-\alpha_{f}(\lambda+e) \}$, denoted by $a_1$ and $a_2$ respectively. By lemma \ref{lemma_dot_product_of_Fouriercoefficients}, we have that $\sum_{\lambda \in \mathbb{Z}_{q}} \alpha_{f}(\lambda) \cdot \overline{\alpha_{f}(\lambda+e)}=0$, so $a_1=a_2$. Argument concerning the symmetry of $A_1$ and $A_2$ shows that $2 \mid a_1$ and $2 \mid a_2$, so $q=a_1+a_2=2a_1$ is a multiple of $4$, which leads to a contradiction since $q=2N$ with $N$ odd. This is essentially due to Ikeda \cite{Ikeda}.

When not all prime divisors of $N$ are self-conjugate $\mathrm{mod} \ N$, but $\frac{\alpha_{f}(\lambda+e)}{\alpha_{f}(\lambda)}$ do not have too many possible values, we may still partition elements of $\mathbb{Z} / q \mathbb{Z}$ into different sets according to the value of $\frac{\alpha_{f}(\lambda+e)}{\alpha_{f}(\lambda)}$, consider the two arguments as above and derive a contradiction, therefore prove nonexistence results of certain types.

PARI/GP codes that confirm relevant assertions in this section are available at the file ``codes for section 5'' at the website mentioned at the end of section \ref{section_introduction}.

\subsection{\texorpdfstring{Nonexistence of type $[1,42]$ GBF}{Nonexistence of type [1,42] GBF}} \label{subsection_51}

In this subsection, let $K=\mathbb{Q}(\zeta_{21})$, where we fix $\zeta_{21}=e^{2 \pi \mathrm{i} / 21}$ and $\zeta_{42}=-\zeta_{21}$. Let $x=(1-\zeta_{21})(1-\zeta_{21}^{-1}) \in K^{+}$. By the table in the appendix of \cite{IntroductionToCyclotomicFields}, we know that $h(K^{+})=1$. By the same tables or \cite[Thm 11.1]{IntroductionToCyclotomicFields} we know that $h(K)=1$. Let $\tau \in \mathrm{Gal}(K / \mathbb{Q})$ denote the map of complex conjugation.

By fact \ref{factorization_rules_of_primes_in_cyclotomic_fields}, we have that 
$$7 \mathcal{O}_{K}={(\mathfrak{p}_1)}^{6}  {(\mathfrak{p}_2)}^{6}; \ 2 \mathcal{O}_K= \mathfrak{p}_3 \mathfrak{p}_4 \ \text{and} \ 3\mathcal{O}_K=\mathfrak{p}_{5}^{2}$$
where $\mathfrak{p}_1, \cdots, \mathfrak{p}_{5}$ are distinct prime ideals in $\mathcal{O}_K$. By fact \ref{QuadraticContained}, $\mathbb{Q}(\sqrt{-3})$ and $\mathbb{Q}(\sqrt{-7})$ are contained in $\mathbb{Q}(\zeta_3)$ and $\mathbb{Q}(\zeta_7)$ respectively, so they are contained in K. Hence $\frac{1 \pm \sqrt{-7}}{2} \in K$ and $\sqrt{-3} \in K$. Then it follows directly by the uniqueness of ideal factorization that $2 \mathcal{O}_K=(\frac{1+\sqrt{-7}}{2})(\frac{1-\sqrt{-7}}{2})$ and $3\mathcal{O}_K={(\sqrt{-3})}^{2}$, with $(\frac{1+\sqrt{-7}}{2})$, $(\frac{1-\sqrt{-7}}{2})$ and $(\sqrt{-3})$ all being prime ideals.

Since there does not exist any integer $s$ such that $7^{s} \equiv -1 \ (\mathrm{mod} \ 3)$, by lemma \ref{DecompositionGroup}, $\tau$ is not in the decomposition group of $7$, which means that $\tau(\mathfrak{p}_1)=\overline{\mathfrak{p}_1}=\mathfrak{p}_2$ and $7 \mathcal{O}_{K}={(\mathfrak{p}_1 \overline{\mathfrak{p}_1})}^{6}$. 
 
Since $h(K)=1$, we have that $\mathfrak{p}_1=(\beta)$ and $7 \mathcal{O}_{K}={(\beta)}^{6}  {(\overline{\beta})}^{6}$ for some $\beta \in \mathcal{O}_K$. A computation in PARI/GP via the command \texttt{bnfisprincipal} shows that one may take \footnote{Note that while the correctness of the command \texttt{bnfisprincipal} relies on the GRH, one may on the other hand, check that the computed $\beta$ has norm $7$, hence must generate a prime ideal lying above $7$. Therefore the correctness of the relevant assertions does not rely on the GRH.}
$$\beta=\zeta_{21}^{9}-\zeta_{21}^7+\zeta_{21}^{5}+\zeta_{21}^{2}.$$ 
Let $u=\frac{7}{{(\beta \overline{\beta})}^{6}}$, then $u$ is a unit in $K^{+}$. Moreover, computation in PARI/GP via the command \texttt{nfeltissquare} shows that $\sqrt{ux} \in K^{+}$, so by lemma \ref{criteria_for_A2} we have that $u=v \overline{v}$ for some $v \in U_K$.\footnote{By remark \ref{remark_construction_of_v}, here we may take $v=\sqrt{u (-\zeta_{21})}=-\zeta_{21}^{11} + 6 \cdot \zeta_{21}^{10} - 12 \cdot \zeta_{21}^{9} + 13 \cdot \zeta_{21}^{8} - 3 \cdot \zeta_{21}^{7} - 12 \cdot \zeta_{21}^{6} + 22 \cdot \zeta_{21}^{5} - 18 \cdot \zeta_{21}^{4} + 6 \cdot \zeta_{21}^{3} + 4 \cdot \zeta_{21}^{2} - 6 \cdot \zeta_{21} + 3$ by a computation in PARI/GP.} Moreover we have that $$\sqrt{-3}=2 \cdot \zeta_{21}^{7}+1; \ \sqrt{-7}=\zeta_{21}^{3}+\zeta_{21}^{6}-\zeta_{21}^{9}+\zeta_{21}^{12}-\zeta_{21}^{15}-\zeta_{21}^{18}$$ by fact \ref{fact_quadratic_represented_by_root_of_unity}.

Now we assume that $f: \mathbb{Z}_{42} \rightarrow \mathbb{Z}_{42}$ is a GBF of type $[1,42]$. So we know that for each $\lambda \in \mathbb{Z}_{42}$, $\alpha_{f}(\lambda):=\sum_{x=0}^{41} \zeta_{42}^{f(x)-\lambda x}$ has absolute value $\sqrt{42}$, i.e. $\alpha_{f}(\lambda) \overline{\alpha_{f}(\lambda)}=42$. The following lemma characterizes solutions to $x \overline{x}=42$ in $\mathcal{O}_K$.

 \begin{lemma} \label{solutions_to_eq_42}
If $x \in \mathcal{O}_K$ and $x \overline{x}=42$, then
$$x=\frac{1 \pm \sqrt{-7}}{2} \cdot \sqrt{-3} \cdot v \cdot \beta^{i} \cdot {(\overline{\beta})}^{6-i} \cdot \zeta_{42}^{j},$$
where $i \in \mathbb{Z}, \ 0 \leq i \leq 6$; $j \in \mathbb{Z}, \ 0 \leq j \leq 41$.
 \end{lemma}

\begin{proof}
Since $3$ satisfies the self-conjugation condition, by lemma \ref{lemma_self_conj_reduce_one_prime_factor} we have that $y:=\frac{x}{\sqrt{-3}} \in \mathcal{O}_K$ and $y \overline{y}=14$. By the uniqueness of prime ideal factorization, we see that exactly one of $\frac{1+\sqrt{-7}}{2}$ and $\frac{1-\sqrt{-7}}{2}$ divides $y$. Let $z=y \Big/ \frac{1 \pm \sqrt{-7}}{2}$, then $z \overline{z}=7$, and it follows that $(z)=\big(\beta^{i} \cdot {(\overline{\beta})}^{6-i} \big)$ for some $0 \leq i \leq 6$. Let $t \in U_K$ be the unit such that $z=t \cdot \beta^{i} \cdot {(\overline{\beta})}^{6-i}$. Then $7=z \overline{z}=t \overline{t} {(\beta \overline{\beta})}^{6}=\frac{7 t \overline{t}}{v \overline{v}}$, so $t \overline{t}=v \overline{v}$. By lemma \ref{lemma_solution_is_unique_up_to_roots_of_unity} and fact \ref{roots_of_unity_in_cyclotomic_fields}, we see that $t=v \cdot \zeta_{42}^{j}$ for some $0 \leq j \leq 41$. Combining all these together completes the proof of the lemma. 
\end{proof}

Let $e=21$, then $e$ is an element of order $2$ in $\mathbb{Z}_{42}$. As mentioned above, we investigate what values may $\frac{\alpha_{f}(\lambda+e)}{\alpha_{f}(\lambda)}$ take.

\begin{lemma} \label{lemma_step_1_of_sets_determination}
Let $g$ be a type $[1,42]$ GBF such that $\alpha_{g}(0)=\frac{1 \pm \sqrt{-7}}{2} \cdot \sqrt{-3} \cdot v \cdot \beta^{i} \cdot {(\overline{\beta})}^{6-i}$ for some $0 \leq i \leq 6$. Then $\alpha_{g}(e)$ must be one of the following:
$$ \alpha_{g}(0) \cdot \{ 1,-1,x_1,x_1^{-1},-x_1,-x_1^{-1},x_2,x_2^{-1},-x_2,-x_2^{-1} \},$$
where $x_1=\beta^{-3} \cdot (\overline{\beta})^{3} \cdot \zeta_{42}^{5}$, $x_2=x_{1}^{2}$. 
\end{lemma}

\begin{proof}
Note that $\alpha_{g}(0)+\alpha_{g}(e)=2 \cdot \displaystyle \sum_{x \ \text{even}} \zeta_{42}^{g(x)}$, so $2 \ \Big| \ \Big(\alpha_{g}(0)+\alpha_{g}(e) \Big)$ in $\mathcal{O}_K$, and it follows that $\frac{1 + \sqrt{-7}}{2} \ \Big| \ \alpha_{g}(e)$ iff $\frac{1 + \sqrt{-7}}{2} \ \Big| \ \alpha_{g}(0)$. Then by lemma \ref{solutions_to_eq_42}, $\alpha_{g}(e)=\frac{1 \pm \sqrt{-7}}{2} \cdot \sqrt{-3} \cdot v \cdot \beta^{k} \cdot {(\overline{\beta})}^{6-k} \cdot \zeta_{42}^{j}$ for some $0 \leq k \leq 6$, $0 \leq j \leq 41$, and $\frac{\alpha_{g}(e)}{\alpha_{g}(0)}=\beta^{k-i} \cdot {(\overline{\beta})}^{i-k} \cdot \zeta_{42}^{j}$ since they have the same $\frac{1 \pm \sqrt{-7}}{2}$ term.

Now $2 \ \big| \ \big( \alpha_{g}(0)+\alpha_{g}(e) \big)$ becomes $\frac{1 \pm \sqrt{-7}}{2} \ \big| \ \Big(\beta^{i} \cdot {(\overline{\beta})}^{6-i}+\beta^{k} \cdot {(\overline{\beta})}^{6-k} \cdot \zeta_{42}^{j} \Big)$, with the plus or minus being opposite to that of $\alpha_{g}(0)$. If $i=k$, then $j=0$ or $j=21$ by lemma \ref{lemma_(1+rootofunity)_divides_prime_above_2}, and $\alpha_{g}(0)=\pm \alpha_{g}(e)$. If $i>k$, then let $l=i-k$, so $0 < l \leq 6$ and we have that $\frac{1 \pm \sqrt{-7}}{2} \ \big| \ \big({\beta}^{l}+{\overline{\beta}}^{l} \zeta_{42}^{j} \big)$. If here $\frac{1 - \sqrt{-7}}{2} \ \big| \ \big({\beta}^{l}+{\overline{\beta}}^{l} \zeta_{42}^{j} \big)$, then taking conjugate yields that $\frac{1+\sqrt{-7}}{2} \ \big| \ \big({\overline{\beta}}^{l}+\beta^{l} \zeta_{42}^{-j} \big)$, and multiplying $\zeta_{42}^{j}$ on the right hand side shows that $\frac{1+\sqrt{-7}}{2} \ \big| \ \big({\beta}^{l}+{\overline{\beta}}^{l} \zeta_{42}^{j} \big)$, so we have that 

$$2 \ \big| \ \big({\beta}^{l}+{\overline{\beta}}^{l} \zeta_{42}^{j} \big) \ \text{and} \ \frac{\alpha_{g}(e)}{\alpha_{g}(0)}=\beta^{-l} \cdot {\overline{\beta}}^{l} \cdot \zeta_{42}^{j}.$$

If $\frac{1 + \sqrt{-7}}{2} \ \big| \ \big({\beta}^{l}+{\overline{\beta}}^{l} \zeta_{42}^{j} \big)$, then an exactly same argument shows that the above result still holds, so the above result holds regardless of the term dividing $\alpha_{g}(0)$.

Similarly, when $i<k$, set $r=k-i$, then $0 < r \leq 6$ and we have that
$$2 \ \big| \ \big({\overline{\beta}}^{r}+\beta^{r} \zeta_{42}^{j} \big) \ \text{and} \ \frac{\alpha_{g}(e)}{\alpha_{g}(0)}=\beta^{r} \cdot {(\overline{\beta})}^{-r} \cdot \zeta_{42}^{j}.$$

Note that for $l>0$,
$$2 \ \big| \ \big({\beta}^{l}+{\overline{\beta}}^{l} \zeta_{42}^{j} \big) \ \text{iff} \ 2 \ \big| \ \big({\beta}^{l} \zeta_{42}^{-j} +{\overline{\beta}}^{l} \big),$$
so let $T_1=\Big\{ (l,j) \ \big| \ 2 \mid \big({\beta}^{l}+{\overline{\beta}}^{l} \zeta_{42}^{j} \big) \Big\}$, $T_2=\Big\{ (r,j) \ \big| \ 2 \mid \big({\overline{\beta}}^{r}+\beta^{r} \zeta_{42}^{j} \big) \Big\}$, we have that $T_2=\big\{ (l,-j) \ \big| \ (l,j) \in T_1 \big\}$. And note that for $(l,j) \in T_1$ and $(l,-j) \in T_2$, the corresponding value of $\frac{\alpha_{g}(e)}{\alpha_{g}(0)}$ are $\beta^{-l} \cdot {\overline{\beta}}^{l} \cdot \zeta_{42}^{j}$ and $\beta^{l} \cdot {(\overline{\beta})}^{-l} \cdot \zeta_{42}^{-j}$, which are reciprocals of each other. Therefore, let $S$ be the set of all possible values of $\frac{\alpha_{g}(e)}{\alpha_{g}(0)}$, then it suffices to find all $(l,j)$ such that $2 \ \big| \ \big({\beta}^{l}+\overline{\beta}^{l} \zeta_{42}^{j} \big)$ to determine $S$. Namely, let $S_{1}=\big\{ \ \beta^{-l} \cdot \overline{\beta}^{l} \cdot \zeta_{42}^{j} \ \big| \ (l,j) \in T_1 \big\}$, then
$$S=(\cup_{x \in S_1} \{ x,x^{-1} \}) \cup \{ 1,-1 \}.$$
Also note that $(l,j) \in T_1$ iff $(l,21+j) \in T_1$, and the corresponding value of $\frac{\alpha_{g}(e)}{\alpha_{g}(0)}$ are opposite numbers of each other, so it only suffices to check $0 \leq j \leq 20$, i.e. let
$$S_{0}=\Big\{ \beta^{-l} \cdot \overline{\beta}^{l} \cdot \zeta_{42}^{j} \ \Big| \ 0 \leq j \leq 20, 0<l \leq 6, 2 \ \big| \ \big({\beta}^{l}+\overline{\beta}^{l} \zeta_{42}^{j} \big) \Big\},$$
then we have that
$$S=(\cup_{x \in S_{0}} \{ x,x^{-1},-x,-x^{-1} \}) \cup \{ 1,-1 \}.$$

By using the \texttt{nfalgtobasis} command in PARI/GP, we obtain the coefficient vector of ${\beta}^{l}+\overline{\beta}^{l} \zeta_{42}^{j}$ represented as the linear combination of an integral basis of $\mathcal{O}_K$. Checking whether $2 \ \big| \ \big({\beta}^{l}+\overline{\beta}^{l} \zeta_{42}^{j} \big)$ is equivalent to checking whether the coefficient vector $\mathrm{mod} \ 2$ is equal to the zero vector of same dimension, which can be done easily in PARI/GP. The computation in PARI/GP shows that for $0<l \leq 6$ and $0 \leq j \leq 20$,
$$2 \ \Big| \ \Big({\beta}^{l}+\overline{\beta}^{l} \zeta_{42}^{j} \Big) \ \text{iff} \ (l,j) \in \big\{ (3,5), (6,10) \big\}.$$
 
Thus $S_{0}=\{ x_1, x_{1}^{2} \}$, where $x_1=\beta^{-3} \cdot (\overline{\beta})^{3} \cdot \zeta_{42}^{5}$, and the lemma follows from $S=(\cup_{x \in S_{0}} \{ x,x^{-1},-x,-x^{-1} \}) \cup \{ 1,-1 \}$.
\end{proof}

Let $S=\{ 1,-1,x_1,x_1^{-1},-x_1,-x_1^{-1},x_2,x_2^{-1},-x_2,-x_2^{-1} \}$ be the set in lemma \ref{lemma_step_1_of_sets_determination}. The following lemma transfers the result in lemma \ref{lemma_step_1_of_sets_determination} into $\frac{\alpha_{f}(\lambda+e)}{\alpha_{f}(\lambda)}$ and determines all possible values of $\frac{\alpha_{f}(\lambda+e)}{\alpha_{f}(\lambda)}$.

\begin{lemma} \label{lemma_step_2_of_sets_determination}
Let $f: \mathbb{Z}_{42} \rightarrow \mathbb{Z}_{42}$ be a GBF. Then $\alpha_{f}(\lambda+e)=\alpha_{f}(\lambda) \cdot s$ for some $s \in S$.
\end{lemma}

\begin{proof}
For a type $[1,42]$ GBF $h$, assume that
$$\alpha_{h}(0)=\frac{1 \pm \sqrt{-7}}{2} \cdot \sqrt{-3} \cdot v \cdot \beta^{i} \cdot {(\overline{\beta})}^{6-i} \cdot \zeta_{42}^{j},$$
then $g(x)=h(x)-j$ is a GBF satisfying the condition of lemma \ref{lemma_step_1_of_sets_determination}, and $\alpha_{g}(\lambda)=\zeta_{42}^{-j} \cdot \alpha_{h}(\lambda)$ for all $\lambda \in \mathbb{Z}_{42}$. So $\alpha_{h}(e)=\zeta_{42}^{j} \cdot \alpha_{g}(e)=\zeta_{42}^{j} \cdot \alpha_{g}(0) \cdot s=\alpha_{h}(0) \cdot s$ for some $s \in S$, which means that the conclusion of lemma \ref{lemma_step_1_of_sets_determination} holds for all type $[1,42]$ GBF.

For a GBF $f$, consider $g(x)=f(x)-\lambda x$, which is also a GBF, so $\alpha_{g}(e)=\alpha_{g}(0) \cdot s$ for some $s \in S$. But $\alpha_{g}(e)=\alpha_{f}(\lambda+e)$, $\alpha_{g}(0)=\alpha_{f}(\lambda)$, so we have that 
$$\alpha_{f}(\lambda+e)=\alpha_{f}(\lambda) \cdot s$$
for some $s \in S$.
\end{proof}

Now we are in a position to prove the nonexistence of type $[1,42]$ GBF. As mentioned above, we partition $\mathbb{Z}_{42}$ into different sets according to the value of $\frac{\alpha_{f}(\lambda+e)}{\alpha_{f}(\lambda)}$. Hence we define the following sets:

\begin{equation*}
\begin{aligned}
    & A_1=\{ \lambda \in \mathbb{Z}_{42} \ | \ \alpha_{f}(\lambda+e)=\alpha_{f}(\lambda) \}, \\
    & A_2=\{ \lambda \in \mathbb{Z}_{42} \ | \ \alpha_{f}(\lambda+e)=-\alpha_{f}(\lambda) \}, \\
    & A_3=\{ \lambda \in \mathbb{Z}_{42} \ | \ \alpha_{f}(\lambda+e)=\alpha_{f}(\lambda) \cdot x_1 \}, \\
    & A_4=\{ \lambda \in \mathbb{Z}_{42} \ | \ \alpha_{f}(\lambda+e)=\alpha_{f}(\lambda) \cdot x_{1}^{-1} \}, \\
    & A_5=\{ \lambda \in \mathbb{Z}_{42} \ | \ \alpha_{f}(\lambda+e)=\alpha_{f}(\lambda) \cdot (-x_1) \}, \\
    & A_6=\{ \lambda \in \mathbb{Z}_{42} \ | \ \alpha_{f}(\lambda+e)=\alpha_{f}(\lambda) \cdot (-x_{1}^{-1}) \}, \\
    & A_7=\{ \lambda \in \mathbb{Z}_{42} \ | \ \alpha_{f}(\lambda+e)=\alpha_{f}(\lambda) \cdot x_2 \}, \\
    & A_8=\{ \lambda \in \mathbb{Z}_{42} \ | \ \alpha_{f}(\lambda+e)=\alpha_{f}(\lambda) \cdot x_{2}^{-1} \}, \\
    & A_9=\{ \lambda \in \mathbb{Z}_{42} \ | \ \alpha_{f}(\lambda+e)=\alpha_{f}(\lambda) \cdot (-x_2) \}, \\
    & A_{10}=\{ \lambda \in \mathbb{Z}_{42} \ | \ \alpha_{f}(\lambda+e)=\alpha_{f}(\lambda) \cdot (-x_{2}^{-1}) \}.
\end{aligned}
\end{equation*}

Let $a_{i}=|A_i|$. Since $\lambda \in A_{j}$ iff $\lambda+e \in A_{j}$ for $j=1,2$, we have that $2 \ | \ a_{j}$ for $j=1,2$. Similarly we have that $a_3=a_4$, $a_5=a_6$, $a_7=a_8$ and $a_9=a_{10}$, and we denote $a_1=2b_1$, $a_2=2b_2$, $a_3=a_4=b_3$, $a_5=a_6=b_4$, $a_7=a_8=b_5$, $a_{9}=a_{10}=b_6$. Since $\sum_{i=1}^{10} a_i=42$, we have that $\sum_{i=1}^{6} b_i = 21$.

By lemma \ref{lemma_dot_product_of_Fouriercoefficients}, we have that $\displaystyle\sum_{\lambda \in \mathbb{Z}_{42}} \alpha_{f}(\lambda) \cdot \overline{\alpha_{f}(\lambda+e)}=0$. By the definition of generalized bent function, $\alpha_{f}(\lambda) \cdot \overline{\alpha_{f}(\lambda)}=42$ for all $\lambda \in \mathbb{Z}_{42}$. Therefore we have that 

\begin{equation} \label{Fourier_coe_innerproduct}
\begin{aligned}
& 42 \cdot a_1-42 \cdot a_2+ 42 \cdot \overline{x_1} \cdot a_3+42 \cdot \overline{x_{1}^{-1}} \cdot a_4- 42 \cdot \overline{x_1} \cdot a_5  -42 \cdot \overline{x_{1}^{-1}} \cdot a_6 \\
&+ 42 \cdot \overline{x_2} \cdot a_7+42 \cdot \overline{x_{2}^{-1}} \cdot a_8 - 42 \cdot \overline{x_2} \cdot a_9-42 \cdot \overline{x_{2}^{-1}} \cdot a_{10}=0.
\end{aligned}
\end{equation}

Note that $\overline{x_{1}}=x_{1}^{-1}$, $\overline{x_{2}}=x_{2}^{-1}$ and substitute $a_i$ by $b_j$, so (\ref{Fourier_coe_innerproduct}) becomes the following equation:
$$84 \cdot b_1-84 \cdot b_2+ 42 \cdot (x_{1}+x_{1}^{-1}) \cdot b_3- 42 \cdot (x_1+x_{1}^{-1}) \cdot b_4+ 42 \cdot (x_2+x_{2}^{-1}) \cdot b_5- 42 \cdot (x_2+x_{2}^{-1}) \cdot b_6=0.$$

Computation in PARI/GP shows that $x_1+x_{1}^{-1}=\frac{4}{7} \cdot \zeta_{21}^{11} + \frac{4}{7} \cdot \zeta_{21}^9 + \frac{4}{7} \cdot \zeta_{21}^8 - \frac{4}{7} \cdot \zeta_{21}^7 - \frac{4}{7} \cdot \zeta_{21}^4 + \frac{8}{7} \cdot \zeta_{21}^2 - \frac{4}{7} \cdot \zeta_{21} - \frac{2}{7}$ and $x_{2}+x_{2}^{-1}=-\frac{2}{7}$. Hence, consider the constant term, we have that
 $$84 \cdot (b_1-b_2)+ 42 \cdot (b_3-b_4) \cdot (-\frac{2}{7}) + 42 \cdot (b_5-b_6) \cdot (-\frac{2}{7}) =0. $$

This shows that $7 \cdot (b_1-b_2)-(b_3-b_4)-(b_5-b_6)=0$, so 
\begin{equation} \label{mod_contradiction}
\sum_{i=1}^{6} b_i \equiv 0 \ (\mathrm{mod} \ 2).
\end{equation}

But $\sum_{i=1}^{6} b_i=21$, a contradiction to (\ref{mod_contradiction}). This proves the following theorem:

\begin{theorem}
Generalized bent function of type $[1,42]$ does not exist.
\end{theorem}

\begin{remark}
We explain why the above theorem is a new result. To the best of our knowledge, the case of $N$ having two prime factors $p_1 \equiv 3 \ (\mathrm{mod} \ 8)$ and $p_2 \equiv 7 \ (\mathrm{mod} \ 8)$ such that at least one of $p_1$ and $p_2$ is not self-conjugate $\mathrm{mod} \ N$ has been considered only in \cite{Feng_Liu_Ma} and \cite{Jiang_Deng}. But both of these two papers show the nonexistence of GBF by showing that there does not exist $\alpha \in \mathcal{O}_K$ such that $\alpha \overline{\alpha}=2N$, while in our case such $\alpha$ exists. For the same reason, all of our results in this section are new. Actually, taking $p_1=3$ and $p_2=7$ in the bound of theorem 3.1 of \cite{Feng_Liu_Ma} yields something strictly less than $1$, indicating that theorem 3.1 of \cite{Feng_Liu_Ma} does not contain our result.
\end{remark}

\subsection{\texorpdfstring{Nonexistence of type $[1,2 \cdot 3^{a} \cdot 7^{b}]$ GBF}{Nonexistence of type [1,2N] GBF with higher power of 3 and 7 dividing N}} \label{subsection_52}

In this subsection, we generalize the nonexistence result in subsection \ref{subsection_51} to the case of arbitrary $a$ and $b$, i.e. let $N=3^{a} \cdot 7^{b}$ where $a,b \in \mathbb{Z}_{>0}$, we will eventually show that GBF of type $[1,2N]$ does not exist.

Let $L=\mathbb{Q}(\zeta_{N})$, where $N=3^{a} \cdot 7^{b}$, $a,b \in \mathbb{Z}_{>0}$. Similar to subsection \ref{subsection_51}, we need to investigate solutions to $x \in \mathcal{O}_L$, $x \overline{x}=2N$ first. Since $3$ is self-conjugate $\mathrm{mod} \ N$, by lemma \ref{lemma_self_conj_reduce_one_prime_factor} we have that $y:=\frac{x}{{(\sqrt{-3})}^{a}} \in \mathcal{O}_L$ and $y \overline{y}=2 \cdot 7^{b}$. We use Schmidt's field descent method to show that such $y$ may be found in a much smaller field.

\begin{lemma} \cite[Thm 3.5]{def_self_conj_Schmidt}
Let $m,n$ be positive integers, and let $m=\prod_{i=1}^{t} p_i^{c_i}$ be the prime power decomposition of $m$. For each prime divisor $q$ of $n$ let
\begin{equation}
m_q=
\begin{cases}
\prod_{p_i \neq q} p_i, 
& \text{if } m \text{ is odd or } q=2, \\[0.5 em]
4 \prod_{p_i \neq 2,q} p_i, 
& \text{if } m \text{ is even.}
\end{cases}
\end{equation}
Let $\mathcal{D}(n)$ be the set of prime divisors of $n$. We define $F(m,n)=\prod_{i=1}^{t} p_i^{b_i}$ to be the minimum multiple of $\prod_{i=1}^{t} p_i$ such that, for every pair
$(i,q)$, $i\in\{1,\ldots,t\}$, $q\in \mathcal{D}(n)$, at least one of the
following conditions is satisfied:
\begin{enumerate}
    \item[(a)] $q=p_i$ and $(p_i,b_i)\neq (2,1)$,
    \item[(b)] $b_i=c_i$,
    \item[(c)] $q\neq p_i$ and $q^{\mathrm{ord}_{m_q}(q)} \not\equiv 1 \pmod{p_i^{b_i+1}}$, where $\mathrm{ord}_{m_q}(q)$ is the multiplicative order of $q$ $\mathrm{mod} \ m_q$.
\end{enumerate}
Then we have that: Let $m,n \in \mathbb{Z}_{>0}$. Assume that $x \overline{x}=n$ for some $x \in \mathbb{Z}[\zeta_{m}]$, then 
$$x \cdot \zeta_{m}^{j} \in \mathbb{Z}[\zeta_{F(m,n)}]$$
for some $j$.
\end{lemma}

Apply the above lemma to our context, where $m=N=3^{a} \cdot 7^{b}$ and $n=2 \cdot 7^{b}$. Then it is easy to verify that if $a=1$, then $F(m,n)=21$ and if $a \geq 2$, then $F(m,n)=63$. For $N=3^{a} \cdot 7^{b}$ and $L=\mathbb{Q}(\zeta_{N})$, by fact \ref{roots_of_unity_in_cyclotomic_fields} we have that $W_L=\langle \zeta_{2N} \rangle$, so we have the following lemma:

\begin{lemma} \label{lemma_find_all_solutions}
Let $N=3^{a} \cdot 7^{b}$ and $L=\mathbb{Q}(\zeta_{N})$; $N_{0}=21$ if $a=1$ and $N_0=63$ if $a \geq 2$. Let $E=\mathbb{Q}(\zeta_{N_0})$. Let $A=\{ y \in \mathcal{O}_L \ | \ y \overline{y}=2 \cdot 7^{b}  \}$ and $A^{\prime}=\{ y \in \mathcal{O}_E \ | \ y \overline{y}=2 \cdot 7^{b}  \}$. Then $A=\{ y \cdot \zeta_{2N}^{j} \ | \ y \in A^{\prime}, 0 \leq j \leq 2N-1 \}$.
\end{lemma}

Now we determine the set $A^{\prime}$ (hence $A$) in the above lemma. If $a=1$, then $A^{\prime}$ is the set of all $y \in \mathcal{O}_K$ such that $y \overline{y}=2 \cdot 7^{b}$, where $K=\mathbb{Q}(\zeta_{21})$. Recall that $2 \mathcal{O}_K=(\frac{1+\sqrt{-7}}{2})(\frac{1-\sqrt{-7}}{2})$ and $7 \mathcal{O}_{K}={(\beta)}^{6}  {(\overline{\beta})}^{6}$, where $(\frac{1 \pm \sqrt{-7}}{2})$, $(\beta)$ and $(\overline{\beta})$ are prime ideals of $\mathcal{O}_K$. Hence
$$y=\frac{1 \pm \sqrt{-7}}{2} \cdot {\beta}^{l} \cdot {\overline{\beta}}^{6b-l} \cdot u$$
for some $u \in U_K$, and since $y \overline{y}=2 \cdot 7^{b}$, we have that $u \overline{u}=v^{b} \cdot {\overline{v}}^{b}$, and by lemma \ref{lemma_solution_is_unique_up_to_roots_of_unity} we have that $u=v^{b} \cdot \zeta_{42}^{j}$ for some $0 \leq j \leq 41$. Therefore when $a=1$, we have that
$$A=\Big\{ \frac{1 \pm \sqrt{-7}}{2} \cdot {\beta}^{l} \cdot {\overline{\beta}}^{6b-l} \cdot v^{b} \cdot \zeta_{2N}^{j} \ \Big| \ 0 \leq l \leq 6b, 0 \leq j \leq 2N-1 \Big\}.$$

Next we consider the case of $a \geq 2$. Let $E=\mathbb{Q}(\zeta_{63})$. Let $K=\mathbb{Q}(\zeta_{21})$; $\beta$ and $v$ be as in lemma \ref{solutions_to_eq_42}. Let $\tau$ denote the map of complex conjugation. As proved above, $(\beta)$ and $(\overline{\beta})$ are the two prime ideals lying above $7$ in $\mathcal{O}_K$, and $7 \mathcal{O}_{K}={(\beta)}^{6}  {(\overline{\beta})}^{6}$. By lemma \ref{lemma_order_mod_p^s}, $2$ is a primitive root $\mathrm{mod} \ 9$, so the order of $2 \ \mathrm{mod} \ 63$ is $\mathrm{lcm}(2 \cdot 3,3)=6$, and $2$ splits into $\frac{\varphi(63)}{6}=6$ prime ideals in $\mathcal{O}_{E}$. By lemma \ref{DecompositionGroup} and the characterization of integers $N$ satisfying $2^{k} \equiv -1 \ (\mathrm{mod} \ N)$ for some $k$ (see page 563 of \cite{FKQ_earliest} for instance), we see that $\tau$ is not in the decomposition group of $2$ in $E$, so the six ideals are $\mathfrak{p}_1, \mathfrak{p}_2, \mathfrak{p}_3$ and their conjugates, i.e. we have that
$$2 \mathcal{O}_{E}=\mathfrak{p}_{1} \mathfrak{p}_2 \mathfrak{p}_3 \overline{\mathfrak{p}_1} \overline{\mathfrak{p}_2} \overline{\mathfrak{p}_3}.$$
 
Without loss of generality we assume that
$$\frac{1+\sqrt{-7}}{2} \mathcal{O}_{E}=\mathfrak{p}_1 \mathfrak{p}_2 \mathfrak{p}_3 \ \text{and} \ \frac{1-\sqrt{-7}}{2} \mathcal{O}_{E}=\overline{\mathfrak{p}_1} \overline{\mathfrak{p}_2} \overline{\mathfrak{p}_3}.$$  
 
By lemma \ref{lemma_order_mod_p^s}, we have that the order of $7 \ \mathrm{mod} \ 9$ is $3$. Since $\mathrm{f}\big((\beta) \ | \ 7 \big)=\mathrm{f}\big((\overline{\beta}) \ | \ 7 \big)=1$ and $[E:K]=3$, we see that $\beta \mathcal{O}_{E}$ and $\overline{\beta} \mathcal{O}_{E}$ are prime ideals in $\mathcal{O}_{E}$, and $7 \mathcal{O}_{E}={(\beta)}^{6} {(\overline{\beta})}^{6}$.

We shall investigate solutions to the equation $x \in \mathcal{O}_{E}, x \overline{x}=2$, which will turn out to be useful in the following context. Note that solution of the equation exists, for example $x=\frac{1+\sqrt{-7}}{2}$. Although there are $2$ prime ideals lying above $2$ in $K$ while there are $6$ in $E$, the following lemma shows that this does not substantially affect the equation $x \overline{x}=2$.

\begin{lemma} \label{lemma_solutions_to_xxbar=2_in_O_L}
Let $x \in \mathcal{O}_{E}$ such that $x \overline{x}=2$. Then $x=\frac{1 \pm \sqrt{-7}}{2} \cdot \zeta_{126}^{j}$. 
\end{lemma}

\begin{proof}
By the table in the appendix of \cite{IntroductionToCyclotomicFields}, we know that $E$ has class number $7$ and $E^{+}$ has class number $1$. Assume that $I$ is a principal ideal such that $I \overline{I}=2 \mathcal{O}_{E}$. If $I$ is not $(\frac{1+\sqrt{-7}}{2})=\mathfrak{p}_{1} \mathfrak{p}_2 \mathfrak{p}_3$ or $(\frac{1-\sqrt{-7}}{2})=\overline{\mathfrak{p}_1} \overline{\mathfrak{p}_2} \overline{\mathfrak{p}_3}$, then $\mathfrak{p}_i$ and $\overline{\mathfrak{p}_i}$ must be in the same class for some $i$. Let $p_i=\mathfrak{p}_i \cap E^{+}$, then since $\mathfrak{p}_i$ is not invariant under the complex conjugate map, which generates $\mathrm{Gal}(E / E^{+})$, we see that the decomposition group $\mathrm{D}(\mathfrak{p}_i / p_i)$ is trivial, hence $p_i \mathcal{O}_E=\mathfrak{p}_i \overline{\mathfrak{p}_i}$. Since $E^{+}$ has class number $1$, $p_i$ is a principal ideal, so $\mathfrak{p}_i \overline{\mathfrak{p}_i}$ is a principal ideal. Then ${\mathfrak{p}_i}^{2}$ is principal, and hence so is $\mathfrak{p}_i$ since $E$ has class number $7$. Let $D$ be the decomposition field of $2$ in $E$, then we know that $[D:\mathbb{Q}]=6$, and $\mathfrak{p}_i \cap \mathcal{O}_D$ is a prime ideal of $\mathcal{O}_D$ lying above $2$. Note that $\mathfrak{p}_i \cap \mathcal{O}_D$ is inert in $E$, so $\mathrm{N}_{E / D}(\mathfrak{p}_i)={(\mathfrak{p}_i \cap \mathcal{O}_D)}^{6}$, so ${(\mathfrak{p}_i \cap \mathcal{O}_D)}^{6}$ is a principal ideal in $\mathrm{Cl}(D)$.

However, with the help of PARI/GP, we will see that this is a contradiction. We use the command \texttt{galoissubcyclo} to get a defining polynomial of $D$, and then use the command \texttt{bnfinit} to compute the information about the class group of $D$, and use the command \texttt{bnfcertify} to make sure that the computation is true without assuming GRH. Note that while the command \texttt{bnfcertify} is quite time-consuming in general, for our field $D$ it runs quite fast. The computation shows that the class group of $D$ is cyclic of order $21$. Then we use the command \texttt{bnfisprincipal} to check the order of the prime ideals of $D$ lying above $2$, and it turns out that any prime lying above $2$ has order $21$ in $\mathrm{Cl}(D)$, a contradiction to ${(\mathfrak{p}_i \cap \mathcal{O}_D)}^{6}$ being principal.

Hence $I=(\frac{1+\sqrt{-7}}{2})$ and $I=(\frac{1-\sqrt{-7}}{2})$ are the only two principal ideals in $\mathcal{O}_{E}$ such that $I \overline{I}=2 \mathcal{O}_{E}$. Hence if $x \in \mathcal{O}_{E}$ is such that $x \overline{x}=2$, then we have that $(x)=(\frac{1 \pm \sqrt{-7}}{2})$, and by fact \ref{conjugate_length1} and fact \ref{roots_of_unity_in_cyclotomic_fields} we see that $x=\frac{1 \pm \sqrt{-7}}{2} \cdot \zeta_{126}^{j}$.
\end{proof}

Now we investigate solutions to $y \in \mathcal{O}_{E}, y \overline{y}=2 \cdot 7^{b}$ as mentioned above. Recall that $7 \mathcal{O}_{E}={(\beta)}^{6} {(\overline{\beta})}^{6}$, with $\beta \mathcal{O}_{E}$ and $\overline{\beta} \mathcal{O}_{E}$ being prime ideals in $\mathcal{O}_{E}$. So it follows that there exists some $0 \leq l \leq 6b$ such that ${\beta}^{l} \cdot {(\overline{\beta})}^{6b-l} \ \big| \ y$, let $z=\frac{y}{{\beta}^{l} \cdot {(\overline{\beta})}^{6b-l}} \in \mathcal{O}_{E}$, we have that $z \overline{z}=2 v^{b} {\overline{v}}^{b}$ and $(z \cdot v^{-b}) \cdot \overline{(z \cdot v^{-b})}=2$. Hence by lemma \ref{lemma_solutions_to_xxbar=2_in_O_L} we have that $z \cdot v^{-b}=\frac{1 \pm \sqrt{-7}}{2} \cdot \zeta_{126}^{j}$, so $y=\frac{1 \pm \sqrt{-7}}{2} \cdot  {\beta}^{l} \cdot {(\overline{\beta})}^{6b-l} \cdot v^{b} \cdot \zeta_{126}^{j}$, i.e. we have the following lemma:

\begin{lemma}
Assume $y \in \mathcal{O}_{E}$ and $y \overline{y}=2 \cdot 7^{b}$. Then
$$y=\frac{1 \pm \sqrt{-7}}{2} \cdot {\beta}^{l} \cdot {(\overline{\beta})}^{6b-l} \cdot v^{b} \cdot \zeta_{126}^{j}$$
for some $l,j$ such that $0 \leq l \leq 6b$, $0 \leq j \leq 125$; $v$ and $\beta$ are as in lemma \ref{solutions_to_eq_42}.
\end{lemma}

Therefore, for $N=3^{a} \cdot 7^{b}$ with $a \geq 2$ and $L=\mathbb{Q}(\zeta_N)$, the set $A$ introduced in lemma \ref{lemma_find_all_solutions} can be characterized as follows:
$$A=\Big\{ \frac{1 \pm \sqrt{-7}}{2} \cdot {\beta}^{l} \cdot {\overline{\beta}}^{6b-l} \cdot v^{b} \cdot \zeta_{2N}^{j} \ \Big| \ 0 \leq l \leq 6b, 0 \leq j \leq 2N-1 \Big\}.$$
Note that this coincides with the result of $a=1$, so we summarize in the following lemma that:

\begin{lemma}
Let $N=3^{a} \cdot 7^{b}$, where $a,b \in \mathbb{Z}_{>0}$ and $L=\mathbb{Q}(\zeta_{N})$. Suppose $x \in \mathcal{O}_L$ and $x \overline{x}=2N$. Then
$$x=\frac{1 \pm \sqrt{-7}}{2} \cdot {(\sqrt{-3})}^{a} \cdot \beta^{l} \cdot {(\overline{\beta})}^{6b-l} \cdot v^{b} \cdot \zeta_{2N}^{j},$$
where $\beta$ and $v$ are as in subsection \ref{subsection_51}; $0 \leq l \leq 6b$; $0 \leq j \leq 2N-1$.
\end{lemma}

Now we have obtained all solutions to $x \in \mathcal{O}_L$, $x \overline{x}=2N$ for $L=\mathbb{Q}(\zeta_{N})$. Assume that $f$ is a GBF of type $[1,2N]$, where $N=3^{a} \cdot 7^{b}$ and let $e=N$, $q=2N$. As in last subsection, we know that for each $\lambda \in \mathbb{Z}_{q}$, $\alpha_{f}(\lambda):=\sum_{x \in \mathbb{Z}_{q}} \zeta_{q}^{f(x)-\lambda x}$ must be such a solution, and we need to investigate the possible values of $\frac{\alpha_{f}(\lambda+e)}{\alpha_{f}(\lambda)}$. Note that $\alpha_{f}(\lambda+e)+\alpha_{f}(\lambda)$ is a multiple of $2$, so by some arguments analogous to lemma \ref{lemma_step_1_of_sets_determination} and \ref{lemma_step_2_of_sets_determination}, we may reduce the problem of deciding the possible values of $\frac{\alpha_{f}(\lambda+e)}{\alpha_{f}(\lambda)}$ to deciding the pairs $(l,j)$ such that $2 \ \big| \ \big({\beta}^{l}+\overline{\beta}^{l} \zeta_{2N}^{j} \big)$. The following lemma shows that, in essential, there are no new possible values of $\frac{\alpha_{f}(\lambda+e)}{\alpha_{f}(\lambda)}$ (or pairs $(l,j)$) compared to the case of $K=\mathbb{Q}(\zeta_{21})$ in last subsection.

\begin{lemma} \label{lemma_setdetermination_for_7b}
All possible values of $\frac{\alpha_{f}(\lambda+e)}{\alpha_{f}(\lambda)}$ are given in the following set:
$$S=(\cup_{i=1}^{2b} \{ x_{1}^{i},x_{1}^{-i},-x_{1}^{i},-x_{1}^{-i} \}) \cup \{ 1,-1 \},$$
where $x_1=\beta^{-3} \cdot (\overline{\beta})^{3} \cdot \zeta_{42}^{5}$ is as in last subsection.
\end{lemma}

\begin{proof}
Let $(l,j)$ be such that $2 \ \big| \ \big({\beta}^{l}+\overline{\beta}^{l} \zeta_{2N}^{j} \big)$ holds, where $1 \leq l \leq 6b$ and $0 \leq j \leq 2N-1$. First we note that similar to lemma \ref{lemma_step_1_of_sets_determination}, for each $l$, there are at most two $j$ such that $(l,j)$ satisfies $2 \ \big| \ \big({\beta}^{l}+\overline{\beta}^{l} \zeta_{2N}^{j} \big)$, and the two corresponding $j_1$ and $j_2$ must satisfy
$$j_1 \equiv j_2+N \pmod{2N}.$$

First we consider the case of $3 \ | \ l$. Let $l=3i$. We prove that $2 \ \big| \ {\beta}^{3i}+{\overline{\beta}}^{3i} \zeta_{42}^{5i}$ by induction. Since $2 \ \big| \ {\beta}^{3}+{\overline{\beta}}^{3} \zeta_{42}^{5}$ and $2 \ \big| \ {\beta}^{6}+{\overline{\beta}}^{6} \zeta_{42}^{10}$, we may assume that $i \geq 3$. If $i=2t$, then by induction hypothesis we have that $2 \ \big| \ {\beta}^{3t}+{\overline{\beta}}^{3t} \zeta_{42}^{5t}$, and squaring yields that $2 \ \big| \ {\beta}^{6t}+{\overline{\beta}}^{6t} \zeta_{42}^{10t}$, so the conclusion holds for $i$. If $i=2t+1$, then by induction hypothesis we have that $2 \ \big| \ {\beta}^{3t}+{\overline{\beta}}^{3t} \zeta_{42}^{5t}$ and $2 \ \big| \ {\beta}^{3(t+1)}+{\overline{\beta}}^{3(t+1)} \zeta_{42}^{5(t+1)}$, multiplying these together yields that
$$2 \ \big| \ {\beta}^{6t+3}+{\overline{\beta}}^{6t+3} \zeta_{42}^{10t+5}+{\beta}^{3t} \cdot {\overline{\beta}}^{3t} \cdot \zeta_{42}^{5t} \cdot ({\beta}^{3}+{\overline{\beta}}^{3} \zeta_{42}^{5}),$$
so $2 \ \big| \ {\beta}^{6t+3}+{\overline{\beta}}^{6t+3} \zeta_{42}^{10t+5}$, the conclusion holds for $i$ again and this finishes the proof.

Now we consider the case of $3 \nmid l$. Assume that $2 \ \big| \ \big({\beta}^{l}+\overline{\beta}^{l} \zeta_{2N}^{j} \big)$, let $l=3m+\delta$ where $\delta=1 \ \text{or} \ 2$. Since there exists $j^{\prime}$ such that $2 \ \big| \ \big({\beta}^{3m}+\overline{\beta}^{3m} \zeta_{2N}^{j^{\prime}} \big)$, we have that $2 \ \big| \ \big({\overline{\beta}}^{l} \cdot \zeta_{2N}^{j} + \beta^{\delta} \cdot \overline{\beta}^{3m} \zeta_{2N}^{j^{\prime}} \big)$, so there exists $j_{0}$ such that $2 \ \big| \ \big({\beta}^{\delta}+{\overline{\beta}}^{\delta} \cdot \zeta_{2N}^{j_0} \big)$, indicating that such cases may be reduced to $l=1 \ \text{or} \ 2$.

If $l=2$, then up to changing $j$ to $j+N$ one may assume that $j$ is even. Let $j=2 j^{\prime}$. Then $2 \ \big| \ \big({\beta}^{2}+2 \cdot \beta \cdot \overline{\beta} \cdot \zeta_{2N}^{j^{\prime}}+{\overline{\beta}}^{2} \zeta_{2N}^{2 j^{\prime}}\big)={\big(\beta+\overline{\beta} \cdot \zeta_{2N}^{j^{\prime}}  \big)}^{2}$, and since $2$ is unramified it follows that $2 \ \big| \ \beta+{\overline{\beta}} \cdot \zeta_{2N}^{j^{\prime}}$. Thus we only need to consider $l=1$.

If $l=1$ and $(l,j)$ is such that $2 \ \big| \ \big({\beta}^{l}+\overline{\beta}^{l} \zeta_{2N}^{j} \big)$, then $2 \ \big| \ \beta+{\overline{\beta}} \cdot \zeta_{2N}^{j}$ and $2 \ \big| \ {(\beta+{\overline{\beta}} \cdot \zeta_{2N}^{j})}^3$, so $2 \ \big| \ {\beta}^{3}+{\overline{\beta}}^{3} \cdot \zeta_{2N}^{3j}+ \beta \cdot \overline{\beta} \cdot \zeta_{2N}^{j} \cdot \big( \beta+\overline{\beta} \cdot \zeta_{2N}^{j} \big)$, hence $2 \ \big| \ {\beta}^{3}+{\overline{\beta}}^{3} \cdot \zeta_{2N}^{3j}$. But $2 \ \big| \ {\beta}^{3}+{\overline{\beta}}^{3} \zeta_{42}^{5}$, so by the uniqueness of $j$, we have that
$$3j \equiv 5 \cdot 3^{a-1} \cdot 7^{b-1} \ \text{or} \ 5 \cdot 3^{a-1} \cdot 7^{b-1}+N \pmod{2N}.$$
If $a=1$, then $3 \ | \ 5 \cdot  7^{b-1}$, a contradiction! If $a \geq 2$, then
$$j \equiv 5 \cdot 3^{a-2} \cdot 7^{b-1} \ \text{or} \ 26 \cdot 3^{a-2} \cdot 7^{b-1} \pmod{2 \cdot 3^{a-1} \cdot 7^{b}},$$
so $j \equiv t \cdot 3^{a-2} \cdot 7^{b-1} \pmod{2N}$, where $t=5,47,89,26,68 \ \text{or} \ 110$. However, none of $\beta+\overline{\beta} \cdot \zeta_{126}^{t}$ is a multiple of $2$ by a direct check in PARI/GP (which can even be done by hand, and note that it suffices to check for $5,47 \ \text{and} \ 89$). See the appendix for the explicit $\mathrm{mod} \ 2$ coordinates of $\beta+\overline{\beta} \cdot \zeta_{126}^{t}$ expressed as $\mathbb{Z}$-linear combination of an integral basis of $\mathbb{Q}(\zeta_{63})$. Combine the two cases together, we have that 
$$2 \ \big| \ \big({\beta}^{l}+\overline{\beta}^{l} \zeta_{2N}^{j} \big) \ \text{iff} \ l=3i \ \text{and} \ \zeta_{2N}^{j}=\pm \zeta_{42}^{5i}.$$
This, combined with an analysis similar to lemma \ref{lemma_step_1_of_sets_determination} (in a word, all possible values of $\frac{\alpha_{f}(\lambda+e)}{\alpha_{f}(\lambda)}$ must come from some $(l,j)$), proves the assertion of the lemma.
\end{proof}

We need the following lemma, whose proof is quite similar to some part of lemma \ref{lemma_setdetermination_for_7b} to derive a contradiction:

\begin{lemma}
Let $K=\mathbb{Q}(\zeta_{21})$ and $x_1=\beta^{-3} \cdot (\overline{\beta})^{3} \cdot \zeta_{42}^{5}$ be as above. Then for all $i \in \mathbb{Z}_{>0}$, there exists $c \in \mathbb{Z}_{>0}$ (possibly depending on $i$) such that $7^{c} \cdot (x_{1}^{i}+x_{1}^{-i}) \in \mathcal{O}_K$ and $7^{c} \cdot (x_{1}^{i}+x_{1}^{-i}) \equiv 2 \pmod 4$ in $\mathcal{O}_K$. Moreover, it is straightforward to verify that for each $n$ such that $n \in \mathbb{Z}$ and $n>c$, $7^{n} \cdot (x_{1}^{i}+x_{1}^{-i})$ is also congruent to $2$ modulo $4$ in $\mathcal{O}_K$.
\end{lemma}

\begin{proof}
By our computation in subsection \ref{subsection_51}, this is already true for $i=1$ and $2$. Now we prove the lemma by induction. Assume that the lemma holds for all $i<k$, we prove that the lemma holds for $i=k$. If $k=2t+1$, then by induction hypothesis we have that $(x_{1}^{t}+x_{1}^{-t}) \cdot 7^{c_1} \equiv 2 \pmod 4$ and $(x_{1}^{t+1}+x_{1}^{-t-1}) \cdot 7^{c_2} \equiv 2 \pmod 4$ for some $c_1, c_2 \in \mathbb{Z}_{>0}$, so $$(x_{1}^{t}+x_{1}^{-t}) \cdot (x_{1}^{t+1}+x_{1}^{-t-1}) \cdot 7^{c_1+c_2}=(x_{1}^{k}+x_{1}^{-k}+x_1+x_{1}^{-1}) \cdot 7^{c_1+c_2} \equiv 4 \pmod{4}.$$
Hence $(x_{1}^{k}+x_{1}^{-k}) \cdot 7^{c_1+c_2+1} \equiv 2 \pmod{4}$ since we have established that $7(x_1+x_{1}^{-1}) \equiv 2 \pmod{4}$ in subsection \ref{subsection_51}. If $k=2t$, then there exists $c$ such that $(x_{1}^{t}+x_{1}^{-t}) \cdot 7^{c} \equiv 2 \pmod 4$, and squaring this yields the desired result, and this completes the proof.
\end{proof}

Now we are in a position to prove the nonexistence result. Define the sets $A_i$ as in subsection \ref{subsection_51}, consider their symmetry and define the $b_i$ as before, where $b_1$ and $b_2$ still corresponds to the first two sets $A_1$ and $A_2$ consisting of $\lambda$ satisfying $\alpha_{f}(\lambda+e)= \pm \alpha_{f}(\lambda)$. Then similar to subsection \ref{subsection_51} we have $\sum b_i=N$, and by lemma \ref{lemma_dot_product_of_Fouriercoefficients} we have the following relation:

\begin{equation} \label{contradiction_equation}
4N \cdot (b_1-b_2)+\sum_{i=1}^{2b} 2N \cdot (x_{1}^{i}+x_{1}^{-i}) \cdot (b_{2i+1}-b_{2i+2})=0.
\end{equation}

We may express $x_{1}^{i}+x_{1}^{-i}$ as the $\mathbb{Q}$-linear combination of $\zeta_{21}^{j}$, where $0 \leq j \leq 11$. By the above lemma, we may choose a large enough positive integer $C$ such that $7^{C} \cdot (x_{1}^{i}+x_{1}^{-i}) \equiv 2 \pmod{4}$ for each $i$ and multiply $7^{C}$ to the above equation. Then by considering the constant term, we have that 
$$4L \cdot (b_1-b_2)+\sum_{i=1}^{2b} 2 \cdot l_i \cdot (4n_i+2) \cdot (b_{2i+1}-b_{2i+2})=0,$$
where $L$ and $l_i$ are odd, $n_i \in \mathbb{Z}$. Eliminate the factor $4$, we have that
$$L \cdot (b_1-b_2)+\sum_{i=1}^{2b} l_i \cdot (2n_i+1) \cdot (b_{2i+1}-b_{2i+2})=0.$$

Reducing $\mathrm{mod} \ 2$ yields $\sum b_i \equiv 0 \pmod 2$, a contradiction to $\sum b_i=N$, and this proves the following theorem:

\begin{theorem} \label{main_thm_2}
Let $N=3^{a} \cdot 7^{b}$, where $a,b \in \mathbb{Z}_{>0}$, then generalized bent function of type $[1,2N]$ does not exist.
\end{theorem}

\section{Some nonexistence results via computational methods} \label{section_scattered_results}

In this section, we consider cases having a similar setting to theorems \ref{thm_two_primes_3_and_5} to \ref{thm_twoprimes_3and1} and case (2) of theorem 4.1 of \cite{FKQ_earliest}, but do not satisfy the condition of $2$ having its maximum possible order. Thus we fail to get a closed-form conclusion, but we may still apply the element partition method and generalize the results in \cite{FKQ_earliest} and \cite{Feng_Liu}. PARI/GP codes of programs used in this section are available at the file ``codes for section 6'' at the website mentioned at the end of section \ref{section_introduction}.

Throughout this section, it is assumed that $N=p_{1}^{e_1} p_{2}^{e_2}$, $(\frac{p_1}{p_2})=-1$, $K=\mathbb{Q}(\zeta_{N})$ and $D$ is the decomposition field of $2$ in $K$. Moreover, it is assumed that $p_2 \equiv 1 \pmod{4}$. Let $f_1$ be the order of $2 \ \mathrm{mod} \ p_1$ and $f_2$ be the order of $2 \ \mathrm{mod} \ p_2$, let $f$ be the order of $2 \ \mathrm{mod} \ p_1 p_2$, then $f=\mathrm{lcm}(f_1,f_2)$. Similar to the analysis in theorem \ref{thm_two_primes_3_and_5}, assume that $p_1$ and $p_2$ are not Wieferich primes, then since $(\frac{p_1}{p_2})=(\frac{p_2}{p_1})=-1$, we have that $p_1 \nequiv 1 \pmod{p_2}$ and $p_2 \nequiv 1 \pmod{p_1}$, so the order of $2 \ \mathrm{mod} \ N$ is $f \cdot p_{1}^{e_1-1} p_{2}^{e_2-1}$, i.e. $D \subseteq \mathbb{Q}(\zeta_{p_1 p_2})$, and we may without loss of generality assume that $N=p_1 p_2$, as implemented in our programs.

\subsection{\texorpdfstring{Case $p_1 \equiv 3 \pmod{8}$, $p_2 \equiv 5 \pmod{8}$}{Case p1 =3 mod 8, p2=5 mod 8}}

We illustrate our idea in detail under the setting of theorem \ref{thm_two_primes_3_and_5}, while the case of the other two theorems in section \ref{direct_consequence_of_EPM} and case (2) of theorem 4.1 of \cite{FKQ_earliest} (theorem 2 of \cite{Lv_Li}) share a similar manner.

Let $N=p_{1}^{e_1} p_{2}^{e_2}$, where $p_1 \equiv 3 \ (\mathrm{mod} \ 8)$ and $p_2 \equiv 5 \ (\mathrm{mod} \ 8)$ are two primes satisfying $(\frac{p_1}{p_2})=-1$. Then by Lemma 2.2 of \cite{FKQ_earliest} and by the analysis in the proof of theorem 4.1 of  \cite{FKQ_earliest}, it suffices to check conditions \ref{condition_1} and \ref{condition_2} in $\mathcal{O}_L$, where $L$ is the unique subfield of $K$ with $[L:D]=2$. We do this by programming in PARI/GP. Note that the correctness of the main command used in our program, \texttt{bnfinit}, is based on GRH. Thus, unless otherwise stated, all the results in this section are based on the assumption of GRH.

Assume that $f$ is a GBF of type $[n,2N]$. Following the notations of \cite{FKQ_earliest}, let $[D:\mathbb{Q}]=2s$ and let $E=\mathbb{Q}(\sqrt{-p_1 p_2})$, then $2 \mathcal{O}_L=\mathfrak{P}_1 \cdots \mathfrak{P}_s \overline{\mathfrak{P}_1} \cdots \overline{\mathfrak{P}_s}$, and to check conditions \ref{condition_1} and \ref{condition_2}, it suffices to check that for each principal ideal $I \subseteq \mathcal{O}_L$ such that $I \overline{I}=2^{n} \mathcal{O}_L$, $(2) \nmid I$ and for $I_1 \overline{I_1}=I_2 \overline{I_2}=2^{n} \mathcal{O}_L$, if $I_1 \neq I_2$, then $I_1$ and $I_2$ do not have the same support. So we need to find all principal ideals $I \subseteq \mathcal{O}_L$ such that $I \overline{I}=2^{n} \mathcal{O}_L$.

First we explain how to identify the ideals $\mathfrak{P}_1, \cdots, \mathfrak{P}_s, \overline{\mathfrak{P}_1}, \cdots, \overline{\mathfrak{P}_s}$ in PARI/GP. First we use the command \texttt{idealprimedec} to obtain all the prime ideals lying above $2$ in $\mathcal{O}_L$ as a list, and assume the first ideal in the list to be $\mathfrak{P}_1$. Then we use the command \texttt{rnfidealdown} to find all ideals $\mathfrak{P}_1, \cdots, \mathfrak{P}_s$ lying above the same prime ideal of $\mathcal{O}_E$ as $\mathfrak{P}_1$. Then we fix a root $\gamma$ of the defining polynomial of $L$ via the command \texttt{L.roots}, obtain all the automorphisms of $L$ via the command \texttt{galoisinit} and \texttt{galoispermtopol}, apply each automorphism to $\gamma$, find the automorphism $\tau$ that $\tau(\gamma)-\overline{\gamma}$ has a numerically very small absolute value and identify it as the complex conjugation. Finally we apply $\tau$ to $\mathfrak{P}_i$ to obtain $\overline{\mathfrak{P}_i}$.

Let $2 \mathcal{O}_E=\mathfrak{p} \overline{\mathfrak{p}}$, and without loss of generality assume that each $\mathfrak{P}_i$ lies over $\mathfrak{p}$. Suppose $I=(\alpha)= \mathfrak{P}_{1}^{n-i_1} \cdots \mathfrak{P}_{s}^{n-i_s} {\overline{\mathfrak{P}_1}}^{i_1} \cdots {\overline{\mathfrak{P}_s}}^{i_s}$, let $k=\sum_{m=1}^{s} (n-i_m)$ and $l=\sum_{m=1}^{s} i_m$. Consider $\mathrm{N}_{L/E}(I)$ we have that $\mathrm{N}_{L/E}(\alpha) \ \mathcal{O}_E={\mathfrak{p}}^{2k} {\overline{\mathfrak{p}}}^{2l}$, since $\mathrm{f}(\mathfrak{P}_i | \mathfrak{p})=\mathrm{f}(\overline{\mathfrak{P}_i} | \overline{\mathfrak{p}})=2$. Hence $2(k-l)$ must be a multiple of the order of $[\mathfrak{p}]$ in $\mathrm{Cl}(E)$, which is $2m$ by \cite{FKQ_earliest}, where $m$ has the same meaning as in theorem \ref{thm_two_primes_3_and_5}. Hence $k-l$ must be a multiple of $m$, and by taking conjugates we may assume that $l \leq k$. Note that $k+l=ns$, so we may solve all possible pairs $(k,l)$ using the above conditions. Then for a fixed pair $(k,l)$, we solve all possible solutions to $\sum_{m=1}^{s} x_m=l$ where $x_m \in \mathbb{Z}$ and $0 \leq x_m \leq n$, let $i_m=x_m$ and check if the corresponding $I=\mathfrak{P}_{1}^{n-i_1} \cdots \mathfrak{P}_{s}^{n-i_s} {\overline{\mathfrak{P}_1}}^{i_1} \cdots {\overline{\mathfrak{P}_s}}^{i_s}$ is principal in $\mathcal{O}_L$ using the command \texttt{bnfisprincipal}. Thus we may obtain all principal ideals $I \subseteq \mathcal{O}_L$ satisfying $I \overline{I}=2^{n} \mathcal{O}_L$, and then check conditions \ref{condition_1} and \ref{condition_2} straightforwardly. We note that, compared to checking all possible values of $(i_1,\cdots,i_s)$ directly, our method of combining some analysis of $E=\mathbb{Q}(\sqrt{-p_1 p_2})$ has better efficiency.

We illustrate some examples. For $p_1=19$, $p_2=13$ and $N=p_{1}^{e_1} p_{2}^{e_2}$, we have that $m=3$ and $s=3$, so theorem 4.1 of \cite{FKQ_earliest} does not give a meaningful bound in this case. Let $2 \mathcal{O}_L=\mathfrak{P}_1 \mathfrak{P}_2 \mathfrak{P}_3 \overline{\mathfrak{P}_1} \overline{\mathfrak{P}_2} \overline{\mathfrak{P}_3}$. Our program shows that solutions to $I \overline{I}=2^{n} \mathcal{O}_L$ are $I={(\mathfrak{P}_1 \mathfrak{P}_2 \mathfrak{P}_3)}^{3}$ and $I={(\overline{\mathfrak{P}_1} \overline{\mathfrak{P}_2} \overline{\mathfrak{P}_3})}^{3}$, so conditions \ref{condition_1} and \ref{condition_2} are satisfied, and there is no GBF of type $[3,2N]$ in this case.

We note that, there appear to be several errors in the tables in examples 2 and 3 of \cite{FKQ_earliest}. For example, setting $p_1=67$, $p_2=13$ and $N=p_{1}^{e_1} p_{2}^{e_2}$, we have that $m=11$ and $s=3$, instead of $s=1$ as claimed in \cite{FKQ_earliest}. So theorem 4.1 of \cite{FKQ_earliest} actually shows the nonexistence of type $[3,2N]$ GBF. On the other hand, our program shows that when $n=11$, the only solutions to $I \overline{I}=2^{n} \mathcal{O}_L$ are $I={(\mathfrak{P}_1 \mathfrak{P}_2 \mathfrak{P}_3)}^{11}$ and $I={(\overline{\mathfrak{P}_1} \overline{\mathfrak{P}_2} \overline{\mathfrak{P}_3})}^{11}$, so there is no GBF of type $[11,2N]$. (hence type $[n,2N]$ for all odd $n \leq 11$, since if there were, then a GBF of type $[11,2N]$ may be constructed. One may also justify this by applying remark \ref{remark_min_solution} to show that there are actually no solutions to $\alpha \in \mathcal{O}_K, \  \alpha \overline{\alpha}=2^{n}$ for $K=\mathbb{Q}(\zeta_{N})$ when $n<11$.)

There are also examples where there are solutions to $\alpha \in \mathcal{O}_K$, $\alpha \overline{\alpha}=2^{n}$, where $n<m$. (Note that there exists $\alpha \in \mathcal{O}_E$ such that $\alpha \overline{\alpha}=2^{m}$.) Let $p_1=11$, $p_2=61$ and $N=p_{1}^{e_1} p_{2}^{e_2}$, then $m=15$ and $s=5$, then our program shows that when $n=13$, there are $10$ principal ideals $I \subseteq \mathcal{O}_L$ such that $I \overline{I}=2^{n} \mathcal{O}_L$, which are conjugates (images under elements in $\mathrm{Gal}(L / \mathbb{Q})$) of one such ideal. Our program shows that conditions \ref{condition_1} and \ref{condition_2} are satisfied in this case, so there is no type $[13,2N]$ GBF. Choose one of these $10$ principal ideals, use \texttt{bnfisprincipal} to recover a generator $\alpha$, one checks directly that $\alpha \overline{\alpha}=2^{13}$.

We conclude all these in the following proposition:

\begin{prop}
Assume GRH, then there is no type $[n,2N]$ GBF where $N=p_{1}^{e_1} p_{2}^{e_2}$ and $(n,p_1,p_2)$ take the following values: $(3,19,13), (11,67,13)$ and $(13,11,61)$.
\end{prop}

\begin{remark}
Essentially, our program solves a PIP problem in a number field of degree $4s$. Since this computational task is exponential in general, when $s \geq 7$, the computation can be quite time-consuming. However, if $s \leq 5$ and $p_1 p_2$ is not too large, the program runs quite rapidly, so we may get more results similar to those in the above proposition.

One may run the command \texttt{bnfcertify} in PARI/GP to remove the assumption of the GRH. However, this command is quite time-consuming in general. Another possible way to remove the GRH assumption is to follow the method of using Stickelberger relations in section 5 of \cite{Lv_Li}, but that method needs extra analysis, so is not very suitable for being programmed directly. On the other hand, the method in \cite{Lv_Li} relies on the fact that $2 \nmid h(\mathbb{Q}(\zeta_{151}))$, but much less is known about the class number of $K=\mathbb{Q}(\zeta_{N})$ in our cases, which increases the difficulty of removing the assumption of the GRH.
\end{remark}

\subsection{\texorpdfstring{Case $p_1 \equiv 7 \pmod{8}$, $p_2 \equiv 5 \pmod{8}$}{Case p1=7 mod 8, p2=5 mod 8}}

In this subsection we present some similar results concerning case (2) of theorem 4.1 of \cite{FKQ_earliest}. We note however that, in this case, the assumption of the GRH can sometimes be removed.

Consider $p_1 \equiv 7 \pmod{8}$, $p_2 \equiv 5 \pmod{8}$ and $N=p_{1}^{e_1} p_{2}^{e_2}$. By results in \cite{FKQ_earliest}, it suffices to check conditions \ref{condition_1} and \ref{condition_2} for solutions to PIP in $\mathcal{O}_D$. Let $E=\mathbb{Q}(\sqrt{-p_1})$ and in this case we use the norm map $\mathrm{N}_{D/E}$ to lessen the cases to be checked. Thus, similar as above, by programming in PARI/GP, we have the following proposition:

\begin{prop}
Assume GRH, then there is no type $[n,2N]$ GBF, where $N=p_{1}^{e_1} p_{2}^{e_2}$ and $(n,p_1,p_2)$ take the following values: $(7,71,61)$, $(5,79,37)$ and $(5,79,61)$. Note that these are cases in the table of example 3 in \cite{FKQ_earliest}, but it is claimed that $s=1$ for all these cases in \cite{FKQ_earliest}, while actually $s=5,3 \ \text{and} \ 3$ respectively.
\end{prop}

Similar to the above remark, we may get more nonexistence results quite rapidly using our program when $p_1 p_2$ is not too large and $s \leq 5$, since we are now working in a field of degree $2s$. But as $s$ gets larger, the program becomes more time-consuming.

In all cases in the above proposition, the conductor of $D$ is $p_1 p_2$. When the conductor of the decomposition field $D$ is $p_2$, we may get some unconditional results:

\begin{prop}
Let $p_1=31$, $p_2 \equiv 5 \pmod{8}$ be a prime, $2^{p_2-1} \nequiv 1 \pmod{p_{2}^{2}}$. Let $N=p_{1}^{e_1} p_{2}^{e_2}$. If $(\frac{p_1}{p_2})=-1$, $2$ is a primitive root $\mathrm{mod} \ p_2$ and $p_2 \nequiv 1 \pmod{5}$ (note that these conditions guarantee that $p_2 \nequiv 1 \pmod{31}$), then there is no GBF of type $[3,2N]$.
\end{prop}

\begin{proof}
We have that $\mathrm{ord}_{31^{e_1}}(2)=5 \cdot 31^{e_1-1}$. Under the assumption of $2^{p_2-1} \nequiv 1 \pmod{p_{2}^{2}}$, we have that $\mathrm{ord}_{p_2}(2)=(p_2-1) \cdot p_{2}^{e_2-1}$, so $\mathrm{ord}_{N}(2)=\mathrm{lcm}(5 \cdot 31^{e_1-1},(p_2-1) \cdot p_{2}^{e_2-1})=5 \cdot 31^{e_1-1} \cdot (p_2-1) \cdot p_{2}^{e_2-1}=\frac{\varphi(N)}{6}$ since $p_2 \neq 5$ (note that $(\frac{31}{5})=1$). Thus we have that $[D:\mathbb{Q}]=6$. Let $D^{\prime}$ be the decomposition field of $2$ in $\mathbb{Q}(\zeta_{31})$. Then we have that $[D: \mathbb{Q}]=[D^{\prime}: \mathbb{Q}]=6$ and $D^{\prime} \subseteq D$, so $D=D^{\prime}$. But solutions to $I \overline{I}=2^{3} \mathcal{O}_{D^{\prime}}$, $I \subseteq \mathcal{O}_{D^{\prime}}$ have been established unconditionally in \cite{Lv_Li}, and conditions \ref{condition_1} and \ref{condition_2} are satisfied, so we see that there is no GBF of type $[3,2N]$.
\end{proof}

Similar results may be established for $p_1=127$ and $p_1=151$. Note that although the case of $p_1=127$ is not directly treated in \cite{Lv_Li}, the method is just the same.

\subsection{\texorpdfstring{Case $p_1 \equiv 7 \pmod{8}$, $p_2 \equiv 1 \pmod{8}$}{Case p1=7 mod 8, p2=1 mod 8}}

For simplicity, we maintain the assumption that $2$ is not a quartic residue $\mathrm{mod} \ p_2$ in this subsection and the following subsection. In this case, by results in \cite{Feng_Liu}, it suffices to check conditions \ref{condition_1} and \ref{condition_2} in $\mathcal{O}_D$. Since $E=\mathbb{Q}(\sqrt{-p_1})$ is contained in $D$, we follow the strategy above and use the norm map $\mathrm{N}_{D/E}$ to lessen cases to be checked. So by running our program, we have the following proposition:

\begin{prop}
Assume GRH, then there is no type $[n,2N]$ GBF, where $N=p_{1}^{e_1} p_{2}^{e_2}$ and $(n,p_1,p_2)$ take the following values: $(3,31,17)$, $(7,71,41)$. Note that the upper bounds of theorem 1 of \cite{Feng_Liu} are $1$ and $\frac{7}{5}$, so our results indeed generalize theorem 1 of \cite{Feng_Liu}.
\end{prop}

\subsection{\texorpdfstring{Case $p_1 \equiv 3 \pmod{8}$, $p_2 \equiv 1 \pmod{8}$}{Case p1=3 mod 8, p2=1 mod 8}}

In this case, by results in \cite{Feng_Liu}, it also suffices to check conditions \ref{condition_1} and \ref{condition_2} in $\mathcal{O}_D$. Since in this case $D$ does not contain an imaginary quadratic subfield, we lessen cases to be checked by considering the norm $\mathrm{N}_{D/T}$, where $T$ is the unique cyclic quartic subfield of $D$. By our program we have the following proposition:

\begin{prop}
Assume GRH, then there is no type $[n,2N]$ GBF, where $N=p_{1}^{e_1} p_{2}^{e_2}$ and $(n,p_1,p_2)$ take the following values: $(23,19,97), (7,11,41), (37,67,97)$. 
\end{prop}

\begin{remark}
We follow the notations of theorem 2 of \cite{Feng_Liu} in the context. Note that by our analysis in the proof of theorem \ref{thm_twoprimes_3and1}, $m$ is just the minimal positive integer $m$ such that there exists a principal ideal $I \subseteq \mathcal{O}_T$ satisfying $I \overline{I}=2^{m} \mathcal{O}_T$. Therefore, we determine $m$ by checking this condition in PARI/GP. We add the command \texttt{bnfcertify} for the field $T$ in the program, so the result is unconditionally correct if \texttt{bnfcertify} returns $1$. Our program shows that for the above three cases, $m=23, 7, 51$ respectively. Therefore the upper bounds of theorem 2 of \cite{Feng_Liu} are $\frac{23}{3}$, $\frac{7}{5}$ and $17$, so our results indeed generalize theorem 2 of \cite{Feng_Liu}.
\end{remark}

\section{Conclusion} \label{conclusion_and_future_work}

In this paper we give some new nonexistence results of GBFs. Both cases where all prime divisors of $N$ are self-conjugate $\mathrm{mod} \ N$ and cases where not all prime divisors of $N$ are self-conjugate $\mathrm{mod} \ N$ are considered. We note that, to the best of our knowledge, the problem of whether some GBFs with quite simple parameters (for example type $[3,14]$) exist, remains unsolved. Thus, the nonexistence problem is far from being completely solved.

\clearpage
\appendix

\section{Numerical results related to theorems in Section~\ref{direct_consequence_of_EPM}}

\begin{remark}
In all three tables, we check the condition that $2$ has multiplicative order $\frac{\varphi(N)}{2}$ (resp $\frac{\varphi(N)}{4}$) $\mathrm{mod} \ N$ for $N=p_1 p_2$. As noted in theorem \ref{thm_two_primes_3_and_5}, if $p_1$ and $p_2$ are not Wieferich primes, then this condition implies that $2$ has multiplicative order $\frac{\varphi(N)}{2}$ (resp $\frac{\varphi(N)}{4}$) $\mathrm{mod} \ N$ for all $N=p_{1}^{e_1} p_{2}^{e_2}$. Note that the two Wieferich primes $1093$ and $3511$ are congruent to $5$ and $7$
modulo $8$ respectively, and $2$ is not a primitive root of $1093$, so this does not affect our counting procedure. For $p=3511$, since
$\mathrm{ord}_{p}(2)=\frac{p-1}{2}$, we simply require that $p_1 \neq 3511$ in
the following tables.
\end{remark}

\begin{table}[H]
\centering
\begin{tabular}{c|r|r|c}
\hline
Upper bound $B$ & \# satisfying condition 1 & \# satisfying both conditions & Ratio \\
\hline
$200$    & $84$       & $51$      & $0.6071428571$ \\
$2000$   & $3061$     & $1280$    & $0.4181639987$ \\
$20000$  & $163755$   & $66604$   & $0.4067295655$ \\
$200000$ & $10159306$ & $4280683$ & $0.4213558485$ \\
\hline
\end{tabular}
\caption{Numerical results for pairs $(p_1,p_2)$ satisfying the conditions in Theorem~1.}
\label{tab:theorem1-numerical-results}

\medskip

Here the upper bound $B$ means that we count pairs $(p_1,p_2)$ such that
$p_1 \leq B$ and $p_2 \leq B$, and the meaning is the same in the following tables.
\end{table}

\begin{table}[htbp]
\centering
\small
\setlength{\tabcolsep}{3pt}
\renewcommand{\arraystretch}{1.1}
\begin{tabular}{c|r|r|r|c|c}
\hline
$B$
& First
& Order
& Non-quartic
& $R_1$
& $R_2$ \\
\hline
$200$    & $58$       & $22$      & $34$      & $0.3793103448$ & $0.6470588235$ \\
$2000$   & $2767$     & $619$     & $1383$    & $0.2237079870$ & $0.4475777296$ \\
$20000$  & $158891$   & $35493$   & $82418$   & $0.2233795495$ & $0.4306462181$ \\
$200000$ & $10108964$ & $2169524$ & $5091168$ & $0.2146138813$ & $0.4261348280$ \\
\hline
\end{tabular}
\caption{Numerical results for pairs $(p_1,p_2)$ related to Theorem~2. Here ``First'' denotes the number of pairs satisfying $p_1 \equiv 7 \pmod{8}$, $p_2 \equiv 1 \pmod{8}$ and $\left(\frac{p_1}{p_2}\right)=-1$; ``Order'' denotes the number of pairs also satisfying the second condition in Theorem~2; and ``Non-quartic'' denotes the number of pairs satisfying the first condition and such that $2$ is not a quartic residue modulo $p_2$, which is just the condition of theorem 1 of \cite{Feng_Liu}. Moreover, $R_1=\mathrm{Order}/\mathrm{First}$ and $R_2=\mathrm{Order}/\mathrm{Non\text{-}quartic}$. Same notations hold for the next table.}
\label{tab:theorem2-nonquartic-numerical-results}
\end{table}

\begin{table}[htbp]
\centering
\small
\setlength{\tabcolsep}{3pt}
\renewcommand{\arraystretch}{1.1}
\begin{tabular}{c|r|r|r|c|c}
\hline
$B$
& First
& Order
& Non-quartic
& $R_1$
& $R_2$ \\
\hline
$200$    & $51$       & $22$      & $29$      & $0.4313725490$ & $0.7586206897$ \\
$2000$   & $2681$     & $553$     & $1330$    & $0.2062663185$ & $0.4157894737$ \\
$20000$  & $159842$   & $35333$   & $82741$   & $0.2210495364$ & $0.4270313388$ \\
$200000$ & $10055849$ & $2149134$ & $5064762$ & $0.2137197963$ & $0.4243306991$ \\
\hline
\end{tabular}
\caption{Numerical results for pairs $(p_1,p_2)$ related to Theorem~3.}
\label{tab:theorem3-nonquartic-numerical-results}
\end{table}

\FloatBarrier

\begin{remark}
As mentioned in Remark~\ref{remark_density_near_primitive_root}, the set $S_2$ of primes
$p_{2} \equiv 1 \pmod{8}$ such that $2$ is not a quartic residue modulo $p_{2}$ has natural
density $\frac{1}{8}$. For a fixed prime $p_1 \equiv 3 \pmod{4}$, the set of primes $p_2 \equiv 1 \pmod{8}$ such that $2$ is quartic residue $\mathrm{mod} \ p_2$ and $\left(\frac{p_1}{p_2}\right)=1$ correspond to the set of primes splitting completely in
$\mathbb{Q}(\mathrm{i}, \sqrt[4]{2}, \sqrt{p_1})$, and hence has natural density
$\frac{1}{16}$ by Chebotarev's density theorem. Moreover, the set of primes $p_2 \equiv 1 \pmod{8}$ such that $(\frac{p_1}{p_2})=1$ corresponds to the set of primes splitting completely in $\mathbb{Q}(\sqrt{2}, \mathrm{i}, \sqrt{p_1})$ and has natural density $\frac{1}{8}$. Therefore, the set of primes $p_2 \equiv 1 \pmod{8}$ such that $2$ is not quartic residue $\mathrm{mod} \ p_2$ and $(\frac{p_1}{p_2})=1$ has natural density $\frac{1}{16}$, and hence the set of primes $p_2 \equiv 1 \pmod{8}$ such that $2$ is not quartic residue $\mathrm{mod} \ p_2$ and $(\frac{p_1}{p_2})=-1$, which is the set of primes in $S_2$ such that $(\frac{p_1}{p_2})=-1$, has density $\frac{1}{16}$. On the other hand, the set of primes $p_2 \equiv 1 \pmod{8}$ such that $\left(\frac{p_1}{p_2}\right)=-1$ has natural density $\frac{1}{8}$, which explains why $R_2$ is roughly $2R_1$.
\end{remark}

\section{\texorpdfstring{Computation of the density of the set of primes $S_1$ in remark \ref{remark_density_near_primitive_root}}{Computation of the density of the set of primes S1 in remark 8}}

We apply theorem 3.1 of \cite{Near_primitive_root} to determine the density of the set of primes $S_1$ defined in remark \ref{remark_density_near_primitive_root}. We recall that, $S_1$ is defined to be the set of primes $p$ such that $p \equiv 1 \pmod{8}$ and $\mathrm{ord}_{p}(2)=\frac{p-1}{2}$. We restate theorem 3.1 of \cite{Near_primitive_root} in the following lemma.

\begin{lemma} \cite[Thm 3.1]{Near_primitive_root} \label{thm:3.1}
Let $1 \leq a \leq d$ be coprime integers. Let $t \geq 1$ be an integer. Put
\[
\mathcal{P}(g,t,d,a)
=
\left\{
p : p \equiv 1 \ (\mathrm{mod} \ t),\ 
p \equiv a \ (\mathrm{mod} \ d),\ 
\operatorname{ord}_p(g)=\frac{p-1}{t}
\right\}.
\]
Let $\sigma_a$ be the automorphism of $\mathbb{Q}(\zeta_d)$ determined by $\sigma_a(\zeta_d)=\zeta_d^a$. Let $c_a(m)$ be $1$ if the restriction of $\sigma_a$ to the field $\mathbb{Q}(\zeta_d) \cap \mathbb{Q}(\zeta_m,g^{1/m})$ is the identity, and let $c_a(m)=0$ otherwise. Put
$$\delta(g,t,d,a)
=
\sum_{n=1}^{\infty}
\frac{\mu(n)c_a(nt)}
{\left[\mathbb{Q}(\zeta_d,\zeta_{nt},g^{1/nt}):\mathbb{Q}\right]}.$$
Then, assuming $\mathrm{RH}$ for all number fields
$\mathbb{Q}(\zeta_d,\zeta_{nt},g^{1/nt})$ with $n$ square-free, we have
$$\mathcal{P}(g,t,d,a)(x)
=
\delta(g,t,d,a)\frac{x}{\log x}
+
O_{g,t,d}\left(\frac{x\log\log x}{\log^2 x}\right).$$
\end{lemma}

By lemma \ref{thm:3.1}, the natural density of $S_1$ is $\delta(2,2,8,1)$, i.e. $t=g=2$, $d=8$ and $a=1$. Denote $\delta=\delta(2,2,8,1)$. By definition, when $a=1$, we have that $c_{a}(m)=1$ for all $m \in \mathbb{Z}_{>0}$. Let $I$ be the set of odd squarefree positive integers. Then $$\delta=
\sum_{r \in I} \frac{\mu(r)}{\left[\mathbb{Q}\left(\zeta_8,\zeta_{2r},2^{1/(2r)}\right):\mathbb{Q}\right]}+\sum_{r \in I} \frac{\mu(2r)}{\left[\mathbb{Q}\left(\zeta_8,\zeta_{4r},2^{1/(4r)}\right):\mathbb{Q}\right]}.$$

It suffices to determine $[\mathbb{Q}(\zeta_8,\zeta_{2r},2^{1/(2r)}):\mathbb{Q}]$ and $[\mathbb{Q}(\zeta_8,\zeta_{4r},2^{1/(4r)}):\mathbb{Q}]$ for $r \in I$. Let $L=\mathbb{Q}(\zeta_8,\zeta_{2r},2^{1/(2r)})$ and $K=\mathbb{Q}(\zeta_{8},\zeta_{2r})$. Then $[K:\mathbb{Q}]=4\varphi(r)$ and $\sqrt{2} \in K$. By proposition 5.27 of \cite{milne2022}, if we show that $2^{1/(2p)} \notin K$ for all primes $p$ dividing $r$, then we have that $[L:K]=r$. But if $2^{(1/2p)} \in K$ for some $p$ dividing $r$, then $2^{1/p} \in K$, so the ramification index of $2$ in $K$ would be a multiple of $p$, while we know that it should be $4$ by fact \ref{factorization_rules_of_primes_in_cyclotomic_fields}. Therefore we have that $[L:K]=r$ and hence $[L:\mathbb{Q}]=4r\varphi(r)$. Similarly, let $M=\mathbb{Q}(\zeta_8,\zeta_{4r},2^{1/(4r)})$. If we show that $2^{1/(2p)} \notin K$ for all primes $p$ dividing $2r$, then $[M:K]=2r$ by the same proposition as above. It suffices to check that $\sqrt[4]{2} \notin K$. Assume the converse, then $\mathbb{Q}(\sqrt[4]{2},\mathrm{i}) \subseteq K$, which is a contradiction, since $\mathrm{Gal}(\mathbb{Q}(\sqrt[4]{2},\mathrm{i}) / \mathbb{Q})$ is not abelian, while $\mathrm{Gal}(K / \mathbb{Q})$ is. Hence $[M:K]=2r$ and $[M:\mathbb{Q}]=8r\varphi(r)$. Hence we have that
$$\delta=
\sum_{r \in I} \frac{\mu(r)}{4r\varphi(r)}+\sum_{r \in I} \frac{-\mu(r)}{8r\varphi(r)}=\sum_{r \in I} \frac{\mu(r)}{8r\varphi(r)}.$$
Since $\sum_{r \in I} \frac{\mu(r)}{r\varphi(r)}=2A$, it follows that $\delta=\frac{1}{8} \cdot 2A=\frac{A}{4}$, which confirms that $S_1$ has density $\frac{A}{4}$ under GRH.

\section{\texorpdfstring{Explicit form of $\beta+\overline{\beta} \cdot \zeta_{126}^{t}$ in lemma \ref{lemma_setdetermination_for_7b}}{Explicit form of some elements used in lemma 17}}

\begin{table}[H]
\centering
\small
\caption{PARI/GP verification used in Lemma~\ref{lemma_setdetermination_for_7b}.}
\label{tab:lemma17-pari-check}
\begin{tabular}{c c p{0.68\textwidth}}
\toprule
$t$ & $2 \mid \beta+\overline{\beta}\zeta_{126}^{t}$? 
& Nonzero coordinate positions modulo $2$ \\
\midrule
$5$ 
& No 
& $\{2,7,9,12,27,28,29,34\}$ \\

$47$ 
& No 
& $\{2,7,9,12,25,26,27,32,33,35,36\}$ \\

$89$ 
& No 
& $\{2,7,9,12,25,26,28,29,32,33,34,35,36\}$ \\

\bottomrule
\end{tabular}
\end{table}

The PARI/GP basis-coordinate vectors for these three values of $t$ are as follows:
\[
\begin{aligned}
\mathbf c_{5}
&=
(0,-1,0,0,0,0,1,0,1,0,0,1,
0,0,0,0,0,0,0,0,0,0,0,0,
0,0,1,1,-1,0,0,0,0,1,0,0)^{T},\\[2mm]
\mathbf c_{47}
&=
(0,-1,0,0,0,0,1,0,1,0,0,1,
0,0,0,0,0,0,0,0,0,0,0,0,
1,1,1,0,0,0,0,-1,-1,0,1,-1)^{T},\\[2mm]
\mathbf c_{89}
&=
(0,-1,0,0,0,0,1,0,1,0,0,1,
0,0,0,0,0,0,0,0,0,0,0,0,
-1,-1,-2,-1,1,0,0,1,1,-1,-1,1)^{T}.
\end{aligned}
\]
Here
\[
  \mathbf c_t
  =
  \operatorname{Coord}_{\mathcal B}
  \bigl(\beta+\overline{\beta}\zeta_{126}^{t}\bigr),
\]
where $\mathcal B$ denotes the integral basis of
$\mathcal O_{\mathbb Q(\zeta_{63})}$ used by PARI/GP. Reducing these vectors
modulo $2$ gives the nonzero coordinate positions listed in
Table~\ref{tab:lemma17-pari-check}.

\section*{Statements and Declarations}

\textbf{Funding}
This work is supported by National Key R\&D Program of China (No. 2025YFA1017203).

\noindent \textbf{Conflict of interest}
The authors have no relevant financial or non-financial interests to disclose.

\noindent \textbf{Data availability}
No datasets were generated or analysed during the current study.

\noindent \textbf{Code availability}
The PARI/GP codes supporting the computational parts of this paper are available at
\url{https://github.com/Bluedaydreaming-Y/Codes-for-nonexistence-results-of-generalized-bent-functions}.

\bibliographystyle{plain}
\bibliography{reference}

\end{document}